\documentclass{amsart}

\usepackage{amssymb}
\usepackage{graphicx}

\newtheorem{theorem}{Theorem}[section]
\newtheorem{lemma}[theorem]{Lemma}
\newtheorem{proposition}[theorem]{Proposition}
\newtheorem{corollary}[theorem]{Corollary}

\theoremstyle{definition}
\newtheorem{definition}[theorem]{Definition}

\theoremstyle{remark}
\newtheorem{remark}[theorem]{Remark}

\numberwithin{equation}{section}

\newcommand{\DF}{\ensuremath{\mathcal{E}}}

\DeclareMathOperator{\dom}{dom}%
\DeclareMathOperator{\dist}{dist}%
\DeclareMathOperator{\sppt}{Sppt}

\def\Tri#1#2#3#4{{ 
\put(#1,#2){\line(3,5){#3}} \put(#1,#2){\line(3,0){#4}} \count223=#1 \advance\count223 by #4
\put(\count223,#2){\line(-3,5){#3}}}}

\def\Spic#1#2#3#4{{
\count205=#2\count206=#2\count207=#2\count208=#2\count209=#2
\count210=#1\count211=#3\count214=#3\count212=#4 \divide\count210 by 2 {\ifnum\count210>0
\Spic{\count210}{#2}{#3}{#4} \multiply\count207 by 3 \multiply\count208 by 6 \multiply\count209 by
5 \multiply\count207 by \count210 \multiply\count208 by \count210 \multiply\count209 by \count210
{\advance\count211 by \count207 \advance\count212 by \count209
\Spic{\count210}{#2}{\count211}{\count212}} {\advance\count214 by \count208
\Spic{\count210}{#2}{\count214}{#4}} \else \multiply\count205 by 3 \multiply\count206 by 6
\Tri{#3}{#4}{\count205}{\count206} \fi }}}

\begin{document}

\title[Smooth bumps and Borel theorem on p.c.f. fracals]{Smooth bumps, a Borel theorem and partitions of smooth functions on p.c.f. fractals.}

\author[Rogers]{Luke G. Rogers}
\address{Department of Mathematics\\University of Connecticut\\Storrs, CT 06269-3009\\U.S.A.}
\email{rogers@math.uconn.edu}

\author[Strichartz]{Robert S. Strichartz}
\address{Department of Mathematics\\Malott Hall\\Cornell University\\Ithaca, NY 14853-4201\\U.S.A.}
\email{str@math.cornell.edu}
\thanks{Research supported in part by the National Science Foundation, Grant DMS-0140194}

\author[Teplyaev]{Alexander Teplyaev }
\address{Department of Mathematics\\University of Connecticut\\Storrs, CT 06269-3009\\U.S.A.}
\email{teplyaev@math.uconn.edu}
\thanks{Research supported in part by the National Science Foundation, Grant DMS-0505622}


\date{}

\begin{abstract}
We provide two methods for constructing smooth bump functions and for smoothly cutting off smooth
functions on fractals, one using a probabilistic approach and sub-Gaussian estimates for the heat
operator, and the other using the analytic theory for p.c.f. fractals and a fixed point argument.
The heat semigroup (probabilistic) method is applicable to a more general class of metric measure
spaces with Laplacian, including certain infinitely ramified fractals, however the cut off
technique involves some loss in smoothness.  From the analytic approach we establish a Borel
theorem for p.c.f.~fractals, showing that to any prescribed jet at a junction point there is a
smooth function with that jet. As a consequence we prove that on p.c.f. fractals smooth functions
may be cut off with no loss of smoothness, and thus can be smoothly decomposed subordinate to an
open cover. The latter result provides a replacement for classical partition of unity arguments in
the p.c.f. fractal setting.
\end{abstract}

\maketitle

\section{Introduction}

Recent years have seen considerable developments in the theory of analysis on certain fractal sets
from both probabilistic and analytic viewpoints \cite{MR1668115,Kigamibook,Strichartzbook}. In this
theory, either a Dirichlet energy form or a diffusion on the fractal is used to construct a weak
Laplacian with respect to an appropriate measure, and thereby to define smooth functions. As a
result the Laplacian eigenfunctions are well understood, but we have little knowledge of other
basic smooth functions except in the case where the fractal is the Sierpinski Gasket
\cite{MR2150975,CalculusII,generalizedefnspaper}. At the same time the existence of a rich
collection of smooth functions is crucial to several aspects of classical analysis, where tools
like smooth partitions of unity, test functions and mollifications are frequently used.  In this
work we give two proofs of the existence of smooth bump functions on fractals, one taking the
probabilistic and the other the analytic approach. The probabilistic result
(Theorem~\ref{Heatbump-heatbumptheorem}) is valid provided the fractal supports a heat operator
with continuous kernel and sub-Gaussian bounds, as is known to be the case for many interesting
examples \cite{MR1668115,MR1023950,MR1701339} that include non-post-critically finite (non-p.c.f.)
fractals such as certain Sierpinski carpets.  In this setting it can also be used to cut off
functions of a certain smoothness class, with some loss of smoothness
(Theorem~\ref{heatbump_smoothedfunctionthm}).  By contrast the analytic method (Theorem
\ref{FixedPoint-smoothbumpexiststheorem}) is applicable to the smaller class of self-similar p.c.f.
fractals with a regular harmonic structure and Dirichlet energy in the sense of Kigami
\cite{Kigamibook}.

For p.c.f. fractals we use our result on the existence of bump functions to prove a Borel-type
theorem, showing that there are compactly supported smooth functions with prescribed jet at a
junction point (Theorem \ref{Borel-borelthm}).  This gives a very general answer to a question
raised in \cite{MR2150975,CalculusII}, and previously solved only for the Sierpinski Gasket
\cite{generalizedefnspaper}.  We remark, however, that even in this special case the results of
\cite{generalizedefnspaper} neither contain nor are contained in the theorem proven here, as the
functions in \cite{generalizedefnspaper} do not have compact support, while those here do not deal
with the tangential derivatives at a junction point.

Finally we apply our Borel theorem to the problem of partitioning smooth functions.  Multiplication
does not generally preserve smoothness in the fractal setting \cite{MR1707752}, so the usual
partition of unity method is not available.  As a substitute for this classical tool we show that a
smooth function can be partitioned into smooth pieces with supports subordinate to a given open
cover (Theorem \ref{Partition-partitionthm}).  This result plays a major role in a forthcoming
paper on a theory of distributions for p.c.f. self-similar fractals \cite{distribspaper}.


\subsection*{Setting}

Let $X$ be a self-similar subset of $\mathbb{R}^{d}$ (or more generally any complete metric space)
in the sense that there is a finite collection of contractive similarities $\{F_{j}\}_{j=1}^{N}$ of
the space and $X$ is the unique compact set satisfying $X=\cup_{j=1}^{N}F_{j}(X)$.  The sets
$F_{j}(X)$ are the $1$-cells, and for a word $w=(w_{1},w_{2},\dotsc ,w_{m})$ of length $m$ we
define $F_{w}=F_{w_{1}}\circ\dotsm\circ F_{w_{m}}$ and call $F_{w}(X)$ an $m$-cell.  If $w$ is an
infinite word then we define $[w]_{m}$ to be its length $m$ truncation and let
$F_{w}(X)=\cap_{m}F_{[w]_{m}}(X)$, which is clearly a point in $X$. The map from infinite words to
$X$ is surjective but not injective, and the points of non-injectivity play an important role in
understanding the connectivity properties of the fractal (see Section 1.6 of \cite{Kigamibook}). In
particular there are critical points of the cover by the $F_{j}$, namely those $x$ and $y$ for
which there are $j\neq k$ in $\{1,\dotsc,N\}$ such that $F_{j}(x)=F_{k}(y)$ (so $F_{j}(x)$ is a
critical value). We call an infinite word $w$ critical if $F_{w}(X)$ is a critical value, and then
call $\tilde{w}$ post-critical if there is $j\in\{1,\dotsc,N\}$ such that $jw$ is critical.  The
boundary $\partial X$ of $X$ consists of all points $F_{\tilde{w}}(X)$ for which $\tilde{w}$ is
post-critical.  In the case that the set of post-critical words is finite the fractal is called
post-critically finite (p.c.f.) and we also use the notations $V_{0}=\partial X$ and
$V_{m}=\cup_{w}F_{w}(V_{0})$, where the union is over all words of length $m$.  The points in
$(\cup_{m}V_{m})\setminus V_{0}$ are called {\em junction points}. We shall always assume that
$V_{0}$ contains at least two elements.

We suppose that $X$ comes equipped with a self-similar probability measure $\mu$, meaning that
there are $\mu_{1},\dotsc,\mu_{N}$ such that the cell corresponding to $w=(w_{1},\dotsc,w_{m})$ has
measure $\mu(F_{w}(X))=\prod_{j=1}^{m}\mu_{w_{j}}$. In order to do analysis on $X$ we assume that
$X$ admits a Dirichlet form \DF, so \DF\ is a closed quadratic form on $L^{2}(\mu)$ with the
(Markov) property that if $u\in\dom(\DF)$ then so is $\tilde{u}=u\chi_{0<u<1} +\chi_{u\geq1}$ and
$\DF(\tilde{u},\tilde{u})\leq\DF(u,u)$, where $\chi_{A}$ is the characteristic function of $A$.  We
will work only with self-similar Dirichlet forms, having the property that
\begin{equation}\label{Setting-DFisselfsimilar}
    \DF(u,v) = \sum_{\text{$m$-words w}} r_{w}^{-1} \DF(u\circ F_{w},v\circ F_{w})
    \end{equation}
where the factors $r_{j}$ are called resistance renormalization factors and as usual
$r_{w}=r_{w_{1}}\dotsm r_{w_{m}}$. For convenience we restrict to the case of regular harmonic
structures, in which $0<r_{j}<1$ for all $j$.  In addition we assume $\DF$ has the property that
$C(X)\cap\dom(\DF)$ is dense both in $\dom(\DF)$ with \DF-norm and in the space of continuous
functions $C(X)$ with supremum norm. We often refer to \DF\ as the energy. If $X$ is a nested
fractal in the sense of Lindstr\o m \cite{MR988082} then such a Dirichlet form may be constructed
using a diffusion or a harmonic structure \cite{MR1442498,MR1301625,MR1474807}.  Other approaches
may be found in \cite{MR1769995,MR2017320,MR2105771,MR2032945}.

Using the energy and measure we produce a weak Laplacian by defining $f=\Delta u$ if $\DF(u,v)=
-\int fv\,d\mu$ for all $v\in\dom(\DF)$ that vanish on $\partial X$.  Then $-\Delta$ is a
non-negative self-adjoint operator on $L^{2}(\mu)$.  When $\Delta u\in C(X)$ we write
$u\in\dom(\Delta)$; this notation is continued inductively to define $\dom(\Delta^{k})$ for each
$k$ and then $\dom(\Delta^{\infty})=\cap_{k} \dom(\Delta^{k})$. We say $f$ is smooth if
$f\in\dom(\Delta^{\infty})$.  On a p.c.f. fractal the weak Laplacian admits an additional pointwise
description in which $\DF$ is a renormalized limit of energies $\DF_{m}$ corresponding to the
finite graph approximations $V_{m}$ and the Laplacian $\Delta$ is a renormalized limit of the
associated graph Laplacians $\Delta_{m}$.  Details are in Section 3.7 of \cite{Kigamibook}.

By standard results, existence of the Dirichlet form \DF\ implies existence of a strongly
continuous semigroup $\{P_{t}\}$ with generator $-\Delta$. Conversely if there is such a semigroup
and it is self-adjoint then there is a corresponding Dirichlet form, so we could equally well begin
with $\{P_{t}\}$ and construct \DF\ (see \cite{MR990239,MR1303354}). The Markov property of \DF\
ensures that if $0\leq u\leq 1$ $\mu$-a.e. then also $0\leq P_{t}u\leq 1$ $\mu$-a.e.

If $X$ is p.c.f. then there is a definition of boundary normal derivatives of a function in
$\dom(\Delta)$ and a Gauss-Green formula relating these to the integral of the Laplacian on $X$.
The usual definition uses resistance-renormalized limits of the terms of the graph Laplacian that
exist at the boundary point.  If $q_{i}$ is the boundary point of $X$ that is fixed by $F_{i}$ and
$r_{i}$ is the resistance factor corresponding to $F_{i}$ we may define a normal derivative
$\partial_{n}$ at $q_{i}$ and have a Gauss-Green formula as in Section 3.7 of \cite{Kigamibook} by
\begin{gather}
    \partial_{n} u(q_{i}) = - \lim_{m\rightarrow\infty} r_{i}^{-m}  \Delta_{m,q_{i}}u(q_{i})
     \label{Introduction-definitionofnormalderiv}\\
    \sum_{q\in\partial X} \bigl( v(q)\partial_{n}u(q) -  u(q)\partial_{n}v(q) \bigr)
    = \int_{X} (v\Delta u- u\Delta v) d\mu \label{Introduction-thegaussgreenformula}
    \end{gather}
where in \eqref{Introduction-definitionofnormalderiv} the quantity $\Delta_{m,q_{i}}u(q_{i})$ is
the graph Laplacian at $q_{i}$.  (We distinguish this from the graph Laplacian at interior points
because it may involve a sum over fewer edges -- for example in the case of the Sierpinski gasket
the graph Laplacian at interior vertices has four terms, but at boundary points there are only
two.)

Normal derivatives may also be localized to cells, so that $\partial_{n,F_{w}(X)}u(F_{w}(q_{i}))$
is given by the limit in \eqref{Introduction-definitionofnormalderiv} but with
$\Delta_{m,w}u(F_{w}(q_{i}))$ denoting the graph Laplacian restricted to to edges that lie inside
$F_{w}(X)$. It is then easy to see that if $\Delta u$ exists and is continuous on each of finitely
many cells that meet at $F_{w}(q_{i})$, then it is continuous on their union if and only if the
following conditions hold: $u$ is continuous, $\Delta u$ has a unique limit at $F_{w}(q_{i})$, and
the normal derivatives at $F_{w}(q_{i})$ sum to zero. We call these the {\em matching conditions}
for the Laplacian.

We shall have need of two other pieces of information about a p.c.f. fractal with regular harmonic
structure.  The first is that there is a Green's function $g(x,y)\geq0$ that is continuous on
$X\times X$ and has self-similar structure related to the discrete Greens function $\Psi(x,y)$ on
$\bigl(V_{1}\setminus V_{0}\bigr)^2$.  According to Section 3.5 of \cite{Kigamibook}
\begin{equation}\label{Introduction-Defnofgreensfn}
    g(x,y) = \lim_{m\rightarrow\infty} \sum_{k=0}^{m-1} \sum_{w\in W_{k}} r_{w}
        \Psi \bigl( F_{w}^{-1}(x),F_{w}^{-1}(y) \bigr)
    \end{equation}
where $W_{k}$ is the collection of words of length $k$ but the sum is only over those $w$ such that
$F_{w}^{-1}(x)$ and $F_{w}^{-1}(y)$ make sense.  Integration against (the negative of) this Green's
function gives the Green's operator $Gf(x)=-\int g(x,y)f(y)d\mu(y)$ and solves $\Delta Gf=f$ with
Dirichlet boundary conditions.  In particular we will make use of pointwise estimates of $g(x,y)$
that follow from \eqref{Introduction-Defnofgreensfn}.  The second thing we need to know is an
estimate on the oscillation of a harmonic function on a cell $F_{w}(X)$, details of which are in
Section 3.2 and Appendix A of \cite{Kigamibook}. A harmonic function $h(x)$ is determined by its
values on the boundary $V_{0}$, and its values may be computed using harmonic extension matrices
$A_{i}$ via $h|_{F_{w}(V_{0})}=A_{w_{m}}\dotsm A_{w_{1}} \cdot h|_{V_{0}}$.  The $A_{i}$ are
positive definite, have eigenvalue $1$ on the constant function and second eigenvalue at most
$r_{i}$.  It follows immediately that the oscillation of $A_{i}h|{V_{0}}$ is at most $r_{i}$ when
the oscillation of $h|_{V_{0}}$ is bounded by $1$, and similarly that the oscillation of
$h|_{F_{w}(X)}$ is at most $r_{w}$.

More details about analysis on self-similar p.c.f. fractals may be found in \cite{Kigamibook}, or
\cite{Strichartzbook} in the special case of the Sierpinski Gasket.  The lecture notes of Barlow
\cite{MR1668115} cover the probabilistic approach that begins with a diffusion semigroup;
non-p.c.f. examples include Sierpinski carpets \cite{MR1023950,MR1701339}. Some of the general
theory connecting Dirichlet forms, heat semigroups and spectral theory of the Laplacian is covered
in \cite{MR990239,MR1303354}.

\subsection*{Smooth bump functions}
In classical analysis on Euclidean spaces the usual bump functions to consider are of the form
$u\in C^{\infty}$ with support in a specified open set $\Omega$, bounds $0\leq u\leq1$ and the
property $u\equiv1$ on a specified compact $K\subset\Omega$.  In the fractal setting described
above it is not usually the case that a product of smooth functions is itself smooth (see
\cite{MR1707752}, or Section \ref{Partition-section} below), so there is less practical benefit to
asking that our bump functions be identically $1$ on $K$ and we will not always do so. Nor is it
always essential that $0\leq u\leq 1$, though this is sometimes useful. For this reason we will use
the term {\em smooth bump function} to mean a function $u\in\dom(\Delta^{\infty})$ with support in
a specified set $\Omega$ and a bound $|u-1|\leq\epsilon$ on a specified compact $K\subset\Omega$.

Suppose $X$ is a p.c.f. self-similar fractal and we have a function $u\in\dom(\Delta^{\infty})$
with $|u-1|\leq \epsilon$ on $K\subset X$ and such that $\Delta^{k}u$ and $\partial_{n}\Delta^{k}u$
vanish on $V_{0}$ for all $k$.  Then for any word $w$ we see that
\begin{equation*}
     u_{w} = \begin{cases}
         u\circ F_{w}^{-1} &\text{on the cell $F_{w}(X)$}\\
         0 &\text{elsewhere}
         \end{cases}
     \end{equation*}
is a smooth bump function with support in $F_{w}(X)$ by the matching conditions for the Laplacian.
For this reason we also use the term {\em smooth bump function} to refer to
$u\in\dom(\Delta^{\infty})$ on $X$ with $|u-1|\leq \epsilon$ on $K$ and which {\em vanishes to
infinite order} at all points $q\in V_{0}$, by which last phrase we mean
$\Delta^{k}u(q)=\partial_{n}\Delta^{k}u(q)=0$ for all $k$.


\section{A smooth bump from the heat operator}\label{Heatbump-section}
In this section $(X,\dist)$ is a metric space with Borel measure $\mu$ and a self-adjoint Neumann
Laplacian $\Delta$.  Our main result here is a procedure for using the heat flow on $X$ to cut-off
a smooth function with only a small loss in smoothness.  For this purpose we require two
assumptions on $\Delta$.

The first assumption is that $\Delta$ has a positive spectral gap, in the sense that there is
$\lambda>0$ such that the spectrum of $\Delta$ is contained in $\{0\}\cup[\lambda,\infty)$. This
implies the estimate $\|P_{t}-I\|_{2,2}\leq \min\{\lambda t,2\}$, where $\|\cdot\|_{2,2}$ refers to
the operator norm on $L^{2}$, $P_{t}$ is the heat operator at time $t$ and $I$ is the identity. For
future use we also note that the spectral representation $P_{t} = \int_{0}^{\infty} e^{-xt}
dE_{\Delta}(x)$ implies there is an estimate
\begin{equation}\label{spectralrepestimate}
    \|\Delta^{k}P_{t}\|_{2,2}\leq c_{k} t^{-k}.
    \end{equation}

For the second assumption, define
\begin{equation*}
    (D(t,d))^{2}
    = \sup \Biggl\{ \int_{L} \bigl| P_{t}f \bigr|^{2} \, :\, \text{$\|f\|_{2}^{2}=1$ and $L\subset
    X$ with $\dist\bigl(\sppt(f), L\bigr)\geq d$} \Biggr\},
    \end{equation*}
so that for any set $L$ and any $f\in L^{2}$ with $\dist\bigl(\sppt(f), L\bigr)\geq d$ we have
\begin{equation}\label{heatbump_heatoperatorestimate}
    \biggl( \int_{L} \bigl| P_{t}f \bigr|^{2}\biggr)^{1/2} \leq D(t,d) \|f\|_{2}.
    \end{equation}
Our second assumption is that there is $\epsilon_{0}$ so that for any $0<\epsilon<\epsilon_{0}$
there is a decreasing sequence $t_{j}$ with all $t_{j}<\frac{2}{\lambda}$ and
$\sum_{j}t_{j}=T<\infty$, and such that for every non-negative integer $k$,
\begin{equation}\label{heatbump_mainassumptiononheatkernel}
    \sum_{j=1}^{\infty} t_{j+1}^{-k} D(t_{j},\epsilon2^{-j}) =C_{k}<\infty.
    \end{equation}

It is worth noting that if this condition is satisfied then we may make $T>0$ as small as we like,
because $D(t,d_{1})\leq D(t,d_{2})$ if $d_{1}\geq d_{2}$ and therefore
\begin{align*}
    \sum_{j=1}^{\infty} t_{j+m+1}^{-k} D(t_{j+m},\epsilon2^{-j})
    &\leq \sum_{j=1}^{\infty} t_{j+m+1}^{-k} D(t_{j+m},\epsilon2^{-(j+m)})\\
    &=\sum_{j'=m}^{\infty} t_{j'+1}^{-k} D(t_{j'},\epsilon2^{-j'})
    \leq C_{k}
    \end{align*}
so we may simply take $m$ so that $\sum_{1}^{\infty} t_{j+m}$ is as small as desired and use the
sequence $t'_{j}=t_{j+m}$.  At the same time, this allows us to reduce a finite number of the
$C_{k}$ to be as small as we desire.  We suppose $C_{0}\leq \frac{1}{2}$.

The second assumption is given in this somewhat artificial form so as to emphasize the estimates
needed in the proof.  It is a non-trivial assumption, but is known to be true in many examples. For
instance, if $P_{t}f$ is given by integration against a kernel $p(t,x,y)$ that satisfies a
sub-Gaussian upper bound
\begin{equation}\label{HeatBump-subGaussianbound}
    p(t,x,y)
    \leq \frac{\gamma_{1}}{t^{\alpha/\beta}}
    \exp\, \biggl( -\gamma_{2} \Bigl( \frac{\dist(x,y)^{\beta}}{t} \Bigr)^{1/(\beta-1)}
    \biggr)
    \end{equation}
then \eqref{heatbump_mainassumptiononheatkernel} may readily verified for sequences $t_{j}$ that
decrease sufficiently rapidly, for example $t_{j}=j^{-j}$ or $t_{j}=e^{-j^{2}}$.  Heat kernel
bounds like \eqref{HeatBump-subGaussianbound} have been the subject of a great deal of research,
not only on fractals but also on manifolds and graphs (see
\cite{MR1098839,MR1150597,MR1853353,MR1938457} and the references therein). Here we satisfy
ourselves with noting that they are known for many interesting examples, including various p.c.f.
fractals \cite{MR1442498,MR1301625,MR1474807,MR1665249} and certain highly symmetric generalized
Sierpinski carpets (which are not p.c.f.) \cite{MR1023950,MR1701339}. They are also valid on
products of some fractals \cite{MR2095624}.  Later we shall see that requiring $p(t,x,y)$ be
continuous and satisfy~\eqref{HeatBump-subGaussianbound} is an appropriate choice when using our
cutoff procedure to produce a smooth function.

\begin{definition}
The space of heat-smoothed functions $\mathcal{A}(X)$ on $X$ is
\begin{equation*}
    \mathcal{A}(X) = \bigcup_{T>0} P_{T}\bigl( L^{2}(X) \bigr)
    \end{equation*}
Functions in $\mathcal{A}$ are characterized by the property that when expanded with respect to an
orthonormal basis of Laplacian eigenfunctions, the coefficients decay exponentially with respect to
the sequence of eigenvalues. In particular $\mathcal{A}(X)$ contains the eigenfunctions and is thus
dense in $L^{2}(X)$.
\end{definition}

\begin{theorem}\label{heatbump_smoothingheatsmoothed}
Let $\phi\in\mathcal{A}(X)$ be $\phi=P_{T}f$ for some $f\in L^{2}(X)$ with $\|f\|_{2}\leq 1$, and
$T>0$. Let $K$ be a fixed compact set. If $0<\epsilon<\epsilon_{0}$ then there is a function $v$
that is equal to $\phi$ on $K$, that vanishes outside the $\epsilon$-neighborhood of $K$, and has
$\Delta^{k}v\in L^{2}$ for all $k$ with $\|\Delta^{k}v\|_{2}\leq
7c_{1}c_{k}C_{k}t_{1}^{-k}e^{2\lambda T}$.
\end{theorem}
On a first reading of the proof, we suggest thinking of the case where $\phi\equiv1$ and $X$ has
finite measure, so that $u_{j}\equiv1$ on $K_{j}$ at each step.  In this case the following
heuristic may be helpful.

Our goal is a function with two properties, the first of which is $L^{2}$-smoothness, meaning that
$\Delta^{k}v\in L^{2}$ for all $k$, and the second is the property of being $\equiv1$ on $K$ and
$\equiv0$ outside the $\epsilon$-neighborhood of $K$, which we call the {\em characteristic}
property. Beginning with a function $u_{1}$ which has the characteristic property but is not
$L^{2}$-smooth, we recursively apply a two step method. The first step smoothes $u_{1}$ by applying
the heat operator for a small time $t_{2}$ to obtain $v_{2}=P_{t_{2}}u_{1}$, which is smooth but
does not have the characteristic property.  The second step splits $X$ into a neighborhood $K_{2}$
of $K$, and the complement $L_{2}$ of a larger neighborhood, as well as the region $A_{2}$ between,
and sets $u_{2}\equiv1$ on $K_{2}$, $u_{2}\equiv0$ on $L_{2}$, and $u_{2}\equiv v_{2}$ on $A_{2}$.

What we have gained in passing from $u_{1}$ to $u_{2}$ is replacing the original abrupt drop of the
characteristic function with the improved piece on $A_{2}$, as illustrated on Figure
\ref{HeatBump-construcitonpic}. This argument is repeated inductively on nested annulus-like
regions $A_{j}$, each time applying the heat kernel for a shorter time $t_{j}$ to get
$v_{j}=P_{t_{j}}u_{j-1}$ and then cutting $v_{j}$ off to be constant outside $A_{j}$ to get
$u_{j}$.

It is unsurprising that this process converges in $L^{2}$. What is perhaps unexpected is that
$\Delta^{k}v$ converges in $L^{2}$ for all $k$, and this is where the
estimate~\eqref{heatbump_mainassumptiononheatkernel} is crucial. Essentially what is going on is
that the ``steepness'' of the $(j+1)$-th interpolant depends both on the height it must interpolate
and the ``width scale'' on which it interpolates. The ``width scale'' depends on $t_{j+1}$ through
the norm $\|\Delta^{k}P_{t_{j+1}}\|_{2,2} \leq c_{k}t_{j+1}^{-k}$ for the ``steepness'' measured by
$\Delta^{k}$, but the height to be interpolated depends instead on how much $P_{t_{j}}$ changed the
function during the smoothing step, which is controlled by $D(t_{j},\epsilon2^{-j})$ because that
is how much $L^{2}$-norm can ``leak'' across the annulus $A_{j}$ when we apply the heat flow.  The
series of terms $\Delta^{k}(v_{j+1}-v_{j})$ is therefore essentially just
$t_{j+1}^{-k}D(t_{j},\epsilon2^{-j})$, which converges
by~\eqref{heatbump_mainassumptiononheatkernel}.

\begin{figure}[htb]
\centerline{
\begin{picture}(200,170)(0,-10)\small%
\put(-20,115){\line(1,0){240}}%
\put(97,130){$u_{1}$}%
{\thicklines \put(-5,115){\line(1,0){25}} \put(20,150){\line(1,0){160}}
\put(20,115){\line(0,1){35}}\put(180,115){\line(0,1){35}} \put(180,115){\line(1,0){25}}}
\put(-20,55){\line(1,0){240}}%
\put(97,70){$v_{2}$}%
{\thicklines \put(0,60){ \qbezier(-1,-4)(15,-4)(23,14) \qbezier(40,30)(30,30)(23,14)}%
\put(200,60){ \qbezier(1,-4)(-15,-4)(-23,14) \qbezier(-40,30)(-30,30)(-23,14)}%
\put(40,90){\qbezier(0,0)(60,4)(120,0)}}%
%
\put(-20,-5){\line(1,0){240}}%
\put(97,10){$u_{2}$}%
{\thicklines \put(12,-5){\qbezier(0,0)(6,0)(10,19) \qbezier(20,35)(16,35)(10,19)}%
\put(188,-5){ \qbezier(0,0)(-6,0)(-10,19) \qbezier(-20,35)(-16,35)(-10,19)}%
\put(32,30){\line(1,0){136}} \put(-7,-5){\line(1,0){19}} \put(188,-5){\line(1,0){19}}}
\end{picture}}
\caption{Initial steps of the proof of Theorem~\protect{\ref{heatbump_smoothingheatsmoothed}} for a constant function.}%
\label{HeatBump-construcitonpic}
\end{figure}
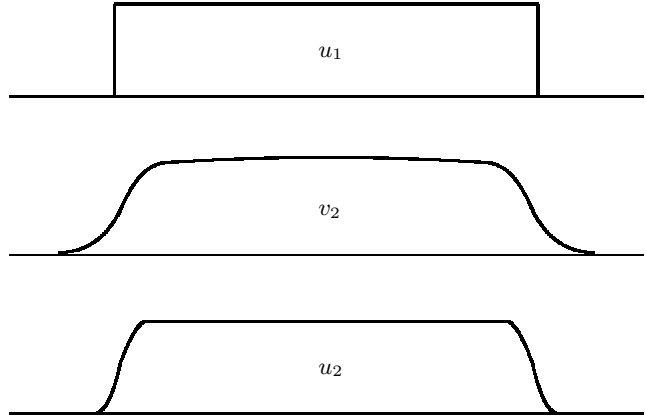
\begin{proof}
Our assumptions imply that we can choose a sequence $t_{j}$ with $\sum_{j}t_{j}=T$ and all
$t_{j}<\frac{2}{\lambda}$ such that~\eqref{heatbump_mainassumptiononheatkernel} holds.

Define neighborhoods $K_{j}$ shrinking dyadically to $K$, neighborhoods $L_{j}$ shrinking
dyadically to the complement of the $\epsilon$-neighborhood of $K$, and annular regions $A_{j}$
between them as follows:
\begin{align*}
    K_{j} &= \{ x: \dist(x,K) < \epsilon 2^{-j} \},\\
    L_{j} &= \{ x: \dist(x,K) > \epsilon (1-2^{-j}) \},\\
    A_{j} &= \{ x: \epsilon 2^{-j}\leq \dist(x,K) \leq \epsilon (1-2^{-j}) \} = X\setminus
    (K_{j}\cup L_{j}).
    \end{align*}
Using these regions, and writing $T_{j}=\sum_{2}^{j}t_{i}$, inductively define for $j\geq 2$
\begin{align*}
    u_{1}&= f \chi_{K_{1}} \\
    v_{j}&= P_{t_{j}} u_{j-1}\\
    u_{j}(x)&= \begin{cases}
            P_{T_{j}}f   &\quad x\in K_{j}\\
            v_{j}(x)&\quad x\in A_{j}\\
            0&\quad x\in L_{j}
            \end{cases}
    \end{align*}
where $\chi_{K_{1}}$ is the characteristic function of the set $K_{1}$.  It will be notationally
convenient to set $t_{1}=0$ and $v_{1}\equiv0$, so that the formula for $u_{j}$ is valid for $j=1$.

From Lemma~\ref{heatbump_seqsareCauchy} we see that $v_{j}$ converges to a function $v$ such that
$\Delta^{k}v\in L^{2}$ for all $k$.  However $u_{j}$ and $v_{j}$ have the same limit by
Lemma~\ref{HeatBump-ujvjhavesamelimitlemma}.  Since $u_{j}\equiv0$ outside the
$\epsilon$-neighborhood of $K$ the same is true of $v$.  Moreover, on $K$ we know that
\begin{equation*}
    u_{j}=P_{T_{j}}f \rightarrow \phi
    \end{equation*}
and we conclude that the function $v$ has the asserted properties.  The estimate for $\Delta^{k}v$
then follows immediately from~\eqref{heatbump_estimateforDeltakofsmoothed} and the
estimate~\eqref{spectralrepestimate}, because the latter gives
\begin{equation*}
    \bigl\| \Delta^{k}v_{2} \bigr\|_{2}
    = \bigl\| P_{t_{1}} u_{1} \bigr\|_{2}
    \leq c_{1}t_{1}^{-k} \|u_{1}\|
    \leq c_{1}t_{1}^{-k}.
    \end{equation*}
\end{proof}

\begin{lemma}\label{heatbump_differenceofvjandujonKj} For $j\geq2$,
\begin{equation*}
    \bigl\| (u_{j}-v_{j})\chi_{K_{j}} \bigr\|_{2}
    \leq D(t_{j},\epsilon 2^{-j}) \bigl( 1 + \|v_{j-1} \|_{2} \bigr).
    \end{equation*}
\end{lemma}
\begin{proof}
From the fact that $P_{t}\phi=e^{-\lambda t}\phi$ we have for $j\geq 2$:
\begin{align*}
    \lefteqn{P_{t_{j}}u_{j-1} - P_{T_{j}}f }\quad\notag\\
    &= P_{t_{j}} \biggl( \Bigr( P_{T_{j-1}}f(x) \Bigl) \chi(x)_{K_{j-1}}
        + v_{j-1}\chi(x)_{A_{j-1}} \biggr) - P_{t_{j}}P_{T_{j-1}} f \notag\\
    &=P_{t_{j}} \biggl( - \Bigr( P_{T_{j-1}}f(x) \Bigl) \chi(x)_{A_{j-1}\cup L_{j-1}}
        + v_{j-1}\chi(x)_{A_{j-1}} \biggr).
    \end{align*}
On $K_{j}$ the above is the difference between $v_{j}$ and $u_{j}$.  Since the support of the
function to which we apply the heat operator is in $A_{j-1}\cup L_{j-1}$ and is therefore at least
$\epsilon 2^{-j}$ distance from $A_{j}$ we can use \eqref{heatbump_heatoperatorestimate} to see
\begin{equation*}
    \int_{K_{j}} |u_{j}-v_{j}|^{2}
    \leq D\bigl(t_{j},\epsilon 2^{-j}\bigr)^{2}
    \bigl\| \Bigr( P_{T_{j-1}}f(x) \Bigl) \chi(x)_{A_{j-1}\cup L_{j-1}} +  v_{j-1}(x)\chi(x)_{A_{j-1}} \bigr\|_{2}^{2}
    \end{equation*}
so
\begin{align*}
    \bigl\| (u_{j}-v_{j})\chi_{K_{j}} \bigr\|_{2}
    &\leq D\bigl(t_{j},\epsilon 2^{-j}\bigr) \bigl\| \Bigr( P_{T_{j-1}}f(x) \Bigl) \chi(x)_{A_{j-1}\cup L_{j-1}} +  v_{j-1}(x)\chi(x)_{A_{j-1}}
    \bigr\|_{2}\\
    &\leq D\bigl(t_{j},\epsilon 2^{-j}\bigr) \bigl( 1 +  \| v_{j-1}(x)\chi(x)_{A_{j-1}} \|_{2}\bigr)
    \end{align*}
because $P_{t}$ contracts $L^{2}$ norm and $\|f\|_{2}\leq 1$.
\end{proof}

\begin{lemma}\label{heatbump_boundforujandvj} For $j\geq 1$
\begin{equation*}
    \|v_{j+1}\|_{2}\leq \| u_{j}\|_{2} \leq 1 + 4\sum_{i=2}^{j} D(t_{i},\epsilon2^{-i}) \leq 3
    \end{equation*}
\end{lemma}
\begin{proof}
Since $P_{t}$ is an $L^{2}$ contraction, the first inequality is immediate.  Note that the second
is true for $j=1$ because $\|u_{1}\|_{2}\leq \|f\|_{2}\leq1$. We also know from our
assumption~\eqref{heatbump_mainassumptiononheatkernel} that
$\sum_{j}D(t_{j},\epsilon2^{-j})=C_{0}\leq\frac{1}{2}$. Induction and
Lemma~\ref{heatbump_differenceofvjandujonKj} then imply
\begin{align*}
    \|u_{j}\|_{2}
    &= \|v_{j}\chi_{A_{j}}\|_{2} + \|u_{j}\chi_{K_{j}}\|_{2} \\
    &\leq \|v_{j}\chi_{A_{j}}\|_{2} + \|v_{j}\chi_{K_{j}}\|_{2} + \|(u_{j}-v_{j})\chi_{K_{j}}\|_{2} \\
    &\leq \|v_{j}\|_{2}+ D(t_{j},\epsilon 2^{-j}) \bigl( 1 + \|v_{j-1} \|_{2} \bigr)\\
    &\leq \|u_{j-1}\|_{2} + D(t_{j},\epsilon 2^{-j}) \bigl( 1 + \|u_{j-2} \|_{2} \bigr)\\
    &\leq 1+ 4\sum_{i=2}^{j-1} D(t_{i},\epsilon2^{-i}) + D(t_{j},\epsilon 2^{-j}) \bigl( 2+
    4\sum_{i=2}^{j-2} D(t_{i},\epsilon2^{-i}) \bigr)\\
    &\leq 1+ 4\sum_{i=2}^{j} D(t_{i},\epsilon2^{-i})
    \leq 3
    \end{align*}
for $j\geq2$.  Of course when $j=2$ the $\|u_{j-2}\|$ term does not appear, and the estimate is
trivial when $j=1$.
\end{proof}

\begin{lemma}\label{HeatBump-ujvjhavesamelimitlemma} For $j\geq2$ we have
$\bigl\|u_{j}-v_{j}\bigr\|_{2}\leq  7 D(t_{j},\epsilon 2^{-j})$.
\end{lemma}

\begin{proof}
Note first that $u_{j}\equiv0$ on $L_{j}$, so
\begin{equation*}
    \bigl\| (u_{j}-v_{j} )\chi_{L_{j}}\bigr\|_{2}^{2} = \bigl\|  v_{j} \chi_{L_{j}}\bigr\|_{2}^{2}
    =\int_{L_{j}}  \bigl| P_{t_{j}} u_{j-1} \bigr|^{2}
    \leq (D(t_{j},\epsilon2^{-j}))^{2} \|u_{j-1}\|_{2}^{2}
    \end{equation*}
because the support of $u_{j-1}$ is in $K_{j-1}\cup A_{j-1}$, so is separated from $L_{j}$ by a
distance of at least $\epsilon2^{-j}$, and therefore~\eqref{heatbump_heatoperatorestimate} is
applicable.

Since $u_{j}$ and $v_{j}$ coincide on $A_{j}$ we have
\begin{align*}
    \bigl\|u_{j}-v_{j}\bigr\|_{2}
    &= \bigl\| (u_{j}-v_{j})\chi_{K_{j}} \bigr\|_{2} +
    \bigl\|(u_{j}-v_{j})\chi_{L_{j}}\bigr\|_{2}\\
    &\leq D(t_{j},\epsilon 2^{-j}) \bigl( 1 + \|v_{j-1} \chi_{A_{j-1}} \|_{2} \bigr) + (D(t_{j},\epsilon2^{-j}))
    \|u_{j-1}\|_{2}\\
    &\leq 7 D(t_{j},\epsilon 2^{-j})
    \end{align*}
by Lemma~\ref{heatbump_differenceofvjandujonKj} and Lemma~\ref{heatbump_boundforujandvj}.
\end{proof}

\begin{lemma}\label{heatbump_seqsareCauchy}
For each non-negative integer $k$ the sequence $\{\Delta^{k}v_{j}\}$ is $L^{2}$-Cauchy.
\end{lemma}

\begin{proof}
The heat operator commutes with $\Delta$, so for $j\geq2$
\begin{align}
    \bigl\|\Delta^{k}(v_{j+1}-v_{j})\bigr\|_{2}
    &=\bigl\| \Delta^{k} ( P_{t_{j+1}} u_{j} - v_{j} ) \bigr\|_{2} \notag\\
    &\leq \bigl\| \Delta^{k} P_{t_{j+1}} (u_{j}-v_{j}) \bigr\|_{2}
        + \bigl\| \Delta^{k} ( P_{t_{j+1}} v_{j}-v_{j} ) \bigr\|_{2} \notag\\
    &\leq \bigl\|\Delta^{k}P_{t_{j+1}} \bigr\|_{2,2} \bigl\|u_{j}-v_{j} \bigr\|_{2}
        + \bigl\| ( P_{t_{j+1}} -I) \Delta^{k}v_{j} \bigr\|_{2} \notag\\
    &\leq 7c_{k}t_{j+1}^{-k} D(t_{j}, 2^{-j}\epsilon)  + \lambda t_{j+1}
\|\Delta^{k}v_{j}\|_{2}. \label{HeatBump-boundfordifferencebetweenvjs}
    \end{align}
where we used~\eqref{spectralrepestimate}, Lemma \ref{HeatBump-ujvjhavesamelimitlemma}, and that
$\|P_{t}-I\|_{2,2}\leq\lambda t$ for $t<\frac{2}{\lambda}$ by the spectral gap assumption. Then
\begin{equation*}
    \|\Delta^{k}v_{j+1}\|_{2} \leq  7 c_{k}t_{j+1}^{-k} D(t_{j}, 2^{-j}\epsilon) +
    (1+\lambda t_{j+1})\|\Delta^{k}v_{j} \|_{2}
    \end{equation*}
and by induction
\begin{align}\label{heatbump_estimateforDeltakofsmoothed}
    \|\Delta^{k} v_{j+1} \|_{2}
    &\leq  7c_{k}\sum_{l=1}^{j} \biggl( t_{l+1}^{-k} D(t_{l}, 2^{-l}\epsilon) \prod_{m=l+1}^{j} (1+\lambda
    t_{m+1}) \biggr) \bigl\| \Delta^{k}v_{2} \bigr\|_{2} \notag\\
    &\leq 7c_{k}C_{k}e^{\lambda T} \bigl\| \Delta^{k}v_{2} \bigr\|_{2}.
    \end{align}
because of~\eqref{heatbump_mainassumptiononheatkernel} and the fact that $\sum_{j}t_{j}\leq T$ so
$\prod_{1}^{\infty} (1+\lambda t_{m+1})\leq e^{\lambda T}$. Substituting into
\eqref{HeatBump-boundfordifferencebetweenvjs} gives
\begin{align*}
    \bigl\|\Delta^{k}(v_{j+1}-v_{j})\bigr\|_{2}
    &\leq 7c_{k} t_{j+1}^{-k} D(t_{j}, 2^{-j}\epsilon)
     + 7c_{k}C_{k}\lambda e^{\lambda T} \bigl\| \Delta^{k}v_{2} \bigr\|_{2}  t_{j+1}
    \end{align*}
which is summable by the assumption~\eqref{heatbump_mainassumptiononheatkernel}.
\end{proof}

The construction in Theorem~\ref{heatbump_smoothingheatsmoothed} produces a function $v$ satisfying
$\Delta^{k}v\in L^{2}$ for all $k$.  The question of whether $v$ is actually smooth (i.e. has
$\Delta^{k}v\in C(X)$ for all $k$) is more difficult.  Given that the only way we have to verify
the assumptions for Theorem~\ref{heatbump_smoothingheatsmoothed} in specific examples is to use an
estimate of the form~\eqref{HeatBump-subGaussianbound}, and that in many instances the heat kernel
is known to be continuous, the following is essentially as useful.

\begin{theorem}\label{heatbump_smoothedfunctionthm}
Suppose that the heat flow $P_{t}$ is given by integration against a continuous kernel $p(t,x,y)$
satisfying~\eqref{HeatBump-subGaussianbound}.  Let $\phi\in\mathcal{A}(X)$ and $K$ be a fixed
compact set.  For any $0<\epsilon<\epsilon_{0}$ there is a smooth function $v$ that is equal to
$\phi$ on $K$ and equal to zero outside the $\epsilon$-neighborhood of $K$.
\end{theorem}

Theorem~\ref{heatbump_smoothedfunctionthm} is proved by applying
Theorem~\ref{heatbump_smoothingheatsmoothed} and the following lemma.

\begin{lemma} If $P_{t}$ is given by integration against a continuous kernel $p(t,x,y)$
satisfying~\eqref{HeatBump-subGaussianbound} then $\Delta^{-k}$ maps $L^{2}(X,\mu)$ into the
continuous functions $C(X)$ for all $k>\frac{\alpha}{\beta}-1$.  In particular, if $f$ is such that
$\Delta^{k}f\in L^{2}$ for all $k$, then $f$ is smooth.
\end{lemma}

\begin{proof}
The spectral representation of $\Delta^{-k}$ implies
\begin{align*}
    \Delta^{-k}f(x)
    &=  \iint_{0}^{\infty} t^{k} p(t,x,y) dt \, f(y) d\mu(y)\\
    &= \iint_{0}^{1} t^{k} p(t,x,y) dt \, f(y) d\mu(y)
        + \iint_{0}^{\infty} (t+1)^{k} p(t+1,x,y) dt \, f(y) d\mu(y)
    \end{align*}
however we may rewrite the second term using
\begin{align*}
    \int_{0}^{\infty} (t+1)^{k} p(t+1,x,y)\, dt
    &=\int_{0}^{\infty} (t+1)^{k} \int p(1,x,z) p(t,z,y)\, d\mu(z)\, dt \\
    &= \int p(1,x,z) \int_{0}^{\infty} (t+1)^{k} p(t,z,y)\, dt \, d\mu(z)\\
    &= (\Delta_{y}+I)^{-k} p(1,x,y).
    \end{align*}
By the inequalities of H\"{o}lder and Minkowski
\begin{align*} 
    \lefteqn{\Bigl| \Delta^{-k}f(x) - \Delta^{-k}f(x') \Bigr| } \quad& \\
    &\leq \|f\|_{2} \biggl\| \int_{0}^{1} t^{k} \bigl( p(t,x,y) - p(t,x',y) \bigr)   dt
    \biggr\|_{L^{2}(y)} \notag\\
    &\qquad +  \|f\|_{2} \biggl\| (\Delta_{y}+I)^{-k} \Bigl( p(1,x,y) -  p(1,x',y) \Bigr) \biggr\|_{L^{2}(y)} \\
    &\leq \|f\|_{2} \int_{0}^{1} t^{k} \bigl\| p(t,x,y) - p(t,x',y) \bigr\|_{L^{2}(y)} \, dt
    + \|f\|_{2} \biggl\|  p(1,x,y) - p(1,x',y) \biggr\|_{L^{2}(y)}
    \end{align*}
because $(\Delta_{y}+I)^{-k}$ is a contraction of $L^{2}$. Now $\bigl\| p(t,x,y) - p(t,x',y)
\bigr\|_{L^{2}(y)}\rightarrow 0$ as $x\rightarrow x'$, for any fixed $t$, because $p(t,x,y)$ is
continuous.  This deals with the second term and shows that the integrand in the first term
converges pointwise to zero. The latter is bounded by
$t^{k}\bigl(\|p(t,x,\cdot)\|_{2}+\|p(t,x',\cdot)\|_{2}\bigr)$, which is an $L^{1}$ function when
$k>\frac{\alpha}{\beta}-1$, because
\begin{equation*}
    \bigl\| p(t,x,y) \bigr\|_{L^{2}(y)}
    = \int p(t,x,y) p(t,x,y)\, d\mu(y)
    = p(2t,x,x)
    \leq C t^{-\alpha/\beta}
    \end{equation*}
by~\eqref{HeatBump-subGaussianbound}.  The result follows by dominated convergence.
\end{proof}

\begin{corollary}\label{Heatbump-heatbumptheorem}
If $\mu(X)<\infty$ then  for any compact set $K$ and neighborhood $U$ of $K$, there is a smooth
non-negative function that is $1$ on $K$ and vanishes outside $U$.
\end{corollary}
\begin{proof} Apply the theorem to the constant function $1$ and note that all of the $u_{j}$ and
$v_{j}$ are non-negative, so the limit function is also.
\end{proof}


\section{A smooth bump as a fixed point of an operator}

To understand why it is sometimes possible to construct a smooth bump function on a self-similar
set as a fixed point of an operator, we invite the reader to consider an elementary situation.  Let
$I=[0,1]$ be the unit interval in $\mathbb{R}$.  We may view $I$ as a p.c.f. self-similar set under
the contractions $f_{0}=x/2$ and $f_{1}=(x+1)/2$.  If $\mu$ is Lebesgue measure and $\mathcal{E}$
is defined using a limit of a regular self-similar harmonic structures with resistance factors
$1/2$ then we obtain the usual Dirichlet energy and Laplacian, and the normal derivatives are the
outward-directed one-sided derivatives at $0$ and $1$ (see \cite{Kigamibook,Strichartzbook} for
details).

The intuition for our construction is as follows.  Consider a symmetric smooth bump function $u$ on
the interval $I=[0,1]$, for which $u\equiv1$ on $[L,1-L]$ as shown in Figure
\ref{FixedPoint-smoothbumponIpic}.
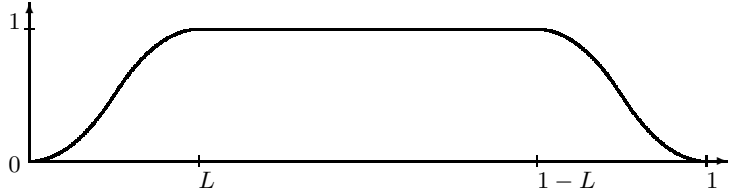
\begin{figure}[htb]
\centerline{
\begin{picture}(264,80)(0,-10)\small
\put(0,0){\vector(1,0){264}} \put(0,0){\vector(0,1){60}} \put(-2,50){\line(1,0){4}}
\put(-8,50){$1$}
\put(-8,-4){$0$} \put(256,-2){\line(0,1){4}} \put(256,-10){$1$} \put(64,-2){\line(0,1){4}}
\put(64,-10){$L$} \put(192,-2){\line(0,1){4}} \put(192,-10){$1-L$}
\thicklines \put(0,0){ \qbezier(0,0)(16,0)(32,25) \qbezier(64,50)(48,50)(32,25)} \put(256,0){
\qbezier(0,0)(-16,0)(-32,25) \qbezier(-64,50)(-48,50)(-32,25)} \put(64,50){\line(1,0){128}}
\end{picture}}
\caption{The smooth bump function $u$.}\label{FixedPoint-smoothbumponIpic}
\end{figure}
If we look at the graph of $\Delta u = d^{2}u/dx^{2}$ we obtain something that looks like a
constant multiple of Figure \ref{FixedPoint-laplacianofbumponIpic}, which appears as if it could be
assembled from rescaled copies of $u$ according to a rule like
\begin{equation}\label{FixedPoint-PhionIdefn}
    \Phi u = \begin{cases}
        u\Bigl( \frac{2x}{L} \Bigr) &\text{ if $0\leq x\leq \frac{L}{2}$}\\
        -u\Bigl( \frac{2x}{L} -1 \Bigr) &\text{ if $\frac{L}{2}< x\leq L$}\\
        0 &\text{ if $L<x<1-L$}\\
        -u\Bigl( \frac{2x-2}{L} +2 \Bigr) &\text{ if $1-L\leq x < 1-\frac{L}{2}$}\\
        u\Bigl( \frac{2x-2}{L} +1 \Bigr) &\text{ if $1-\frac{L}{2}\leq x\leq1$}
        \end{cases}
    \end{equation}
so that we might hope there is actually a smooth bump function $u$ which has precisely this scaling
behavior.  If we let $G$ denote the Green's operator for the operator $\Delta$ on $I$ with
Dirichlet boundary conditions, then this would be equivalent to asking that $u$ be a fixed point of
the operator
\begin{equation}\label{FixedPoint-defnofoperatorPsiforI}
    \Psi u(x) = \frac{G\circ\Phi u(x)}{G\circ\Phi u(1/2)}
    \end{equation}
\begin{figure}[htb]
\centerline{
\begin{picture}(264,120)(0,-60)
\put(0,0){\vector(1,0){264}} \put(0,-60){\vector(0,1){120}} \put(-2,50){\line(1,0){4}}
\put(-8,50){$1$} \put(-2,-50){\line(1,0){4}} \put(-16,-50){${-}1$}
\put(-8,-4){$0$} \put(256,-2){\line(0,1){4}} \put(256,-10){$1$} \put(64,-2){\line(0,1){4}}
\put(64,3){$L$} \put(192,-2){\line(0,1){4}} \put(192,3){$1-L$}
\thicklines \put(64,0){\line(1,0){128}} \multiput(0,0)(224,0){2}{\put(0,0){
\qbezier(0,0)(2,0)(4,25) \qbezier(8,50)(6,50)(4,25)} \put(32,0){ \qbezier(0,0)(-2,0)(-4,25)
\qbezier(-8,50)(-6,50)(-4,25)} \put(8,50){\line(1,0){16}} } \multiput(0,0)(160,0){2}{\put(32,0){
\qbezier(0,0)(2,0)(4,-25) \qbezier(8,-50)(6,-50)(4,-25)} \put(64,0){ \qbezier(0,0)(-2,0)(-4,-25)
\qbezier(-8,-50)(-6,-50)(-4,-25)} \put(40,-50){\line(1,0){16}} }
\end{picture}}
\caption{The function $\Delta u =  d^{2}u/dx^{2} = \Phi(u)$.}
\label{FixedPoint-laplacianofbumponIpic}
\end{figure}
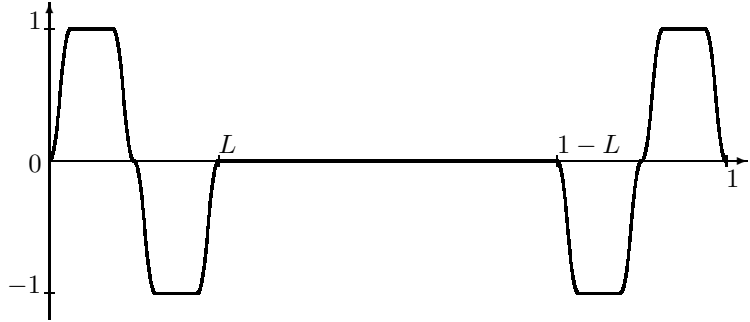
It is a consequence of our general result Theorem \ref{FixedPoint-smoothbumpexiststheorem} that the
operator $\Psi$ in \eqref{FixedPoint-defnofoperatorPsiforI} has a fixed point and that the fixed
point is a smooth bump function.  In fact more is true in the special case of $I$, where the fact
that removing any interior point disconnects the set, along with the existence of an explicit
formula for the Green's function, allows us to prove that the fixed point has values in $[0,1]$ and
is identically $1$ on $[L,1-L]$. For reasons of brevity we do not include the proof of this result;
it is a simpler version of the proof of Theorem \ref{FixedPoint-smoothbumpexiststheorem}.
\begin{proposition}
    If $L$ is sufficiently small then the operator $\Psi$ preserves the
    space of continuous functions on $I$ that have values in $[0,1]$, vanish at $0$ and $1$,
    and are identically $1$ on $[L,1-L]$.  Furthermore $\Psi$ is a contraction in the $L^{\infty}$ norm on these functions,
    and its fixed point is a smooth function that vanishes to infinite order at $0$ and $1$.
    \end{proposition}

Another example in which we can define operators $\Phi$ and $\Psi$ that are similar to
\eqref{FixedPoint-PhionIdefn} and \eqref{FixedPoint-defnofoperatorPsiforI} is the Sierpinski gasket
$SG$ with its standard harmonic structure and measure, where for sufficiently large $l$ we can set
\begin{equation}\label{FixedPoint-defnofPhiforSG}
    \Phi u (x)
    = \begin{cases}
    2u\big(F_{i}^{-(l+1)}(x)\bigl)                   & \text{ if $x\in F_{i}^{ (l+1)}(SG)$}\\
    - u\bigl(F_{j}^{-1}\circ F_{i}^{-l}(x)\bigr)   & \text{ if $x\in F_{i}^{l}\circ F_{j}(SG)$, $j\neq i$}\\
    0                                                     & \text{otherwise}
    \end{cases}
    \end{equation}
and with $p$  any vertex from $V_{1}$ let
\begin{equation}\label{FixedPoint-defnofoperatorPsiforSG}
    \Psi u=\frac{G\circ\Phi u(x)}{G\circ\Phi u(p)}
    \end{equation}
as illustrated in Figure \ref{FixedPoint-phifforSGpic} for the case $l=2$.
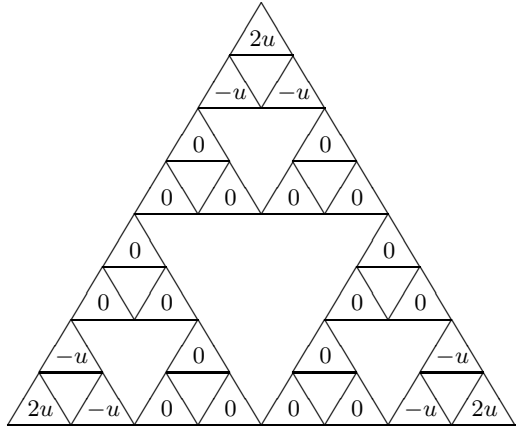
\begin{figure}[htb]
\centerline{
\begin{picture}(192,175)(0,0)
    \setlength{\unitlength}{.5pt}
    \Spic{8}{8}{0}{0}
    \small
    \put(0,0){ \put(15,7){$2u$} \put(60,7){${-}u$} \put(36,47){${-}u$} }
    \put(288,0){ \put(63,7){$2u$} \put(12,7){${-}u$}\put(36,47){${-}u$} }
    \put(144,240){ \put(39,47){$2u$} \put(12,7){${-}u$} \put(60,7){${-}u$} }
    \multiput(96,0)(96,0){2}{ \put(20,7){$0$} \put(44,47){$0$} \put(68,7){$0$} }
    \multiput(96,160)(96,0){2}{ \put(20,7){$0$} \put(44,47){$0$} \put(68,7){$0$} }
    \multiput(48,80)(192,0){2}{ \put(20,7){$0$} \put(44,47){$0$} \put(68,7){$0$} }
\end{picture}}
\caption{$\Phi u$ in the case $l=2$.}\label{FixedPoint-phifforSGpic}
\end{figure}
Again we omit the variant of the proof of Theorem \ref{FixedPoint-smoothbumpexiststheorem} that
establishes the following result
\begin{proposition}
    The operator $\Psi$ of \eqref{FixedPoint-defnofoperatorPsiforSG} is an $L^{\infty}$-contractive self-map of the set of functions that are continuous on $SG$,
vanish at the boundary, are identically $1$ on $SG\setminus\cup_{i} F_{i}^{l}$, and satisfy
$\bigl|\int u -1 \bigr|\leq \frac{1}{2}$. The fixed point of $\,\Psi$ is a smooth bump function.
\end{proposition}

The method described for $I$ and $SG$ rely heavily on the symmetry of these sets and on the
assumption that they are endowed with the symmetrical harmonic structures and measure.  This
assumption is unavoidable if we want to use the same operation $\Phi$ at all steps of the
computation, as the natural linear combination of rescaled copies of the function will not
otherwise have the desired properties, but it is very restrictive.  Even some of the simplest of
the nested fractals defined by Lindstr\o m \cite{MR988082} have insufficient symmetry for a fixed
$\Phi$ to be used in the construction of a smooth bump by this method.  Nonetheless the method can
be adapted to general p.c.f. self-similar fractals with regular harmonic structure and self-similar
measure.

Let $X$ be p.c.f. self-similar with boundary $V_{0}=\partial X$, measure $\mu$ that is self-similar
with scaling factors $\mu_{j}$ and regular harmonic structure with factors $r_j$. We fix a scale
$l_{1}$ with size to be determined later, and label the boundary $l_{1}$ cells by
$Y_{j}=F_{j}^{l_{1}}(X)$. Their union is $Y=\cup Y_{j}$.   For any $\epsilon>0$ we will build a
smooth function that satisfies $|u-1|\leq \epsilon$ on $X\setminus Y$ by a construction that
inductively determines its Laplacian on the cells $Y_{j}$, writing it as a fixed point of an
operator $\Psi$ on the following space of functions.
\begin{definition}
Let $\mathcal{C}$ be the space of continuous functions $u$ on $X$ such that $u(q)=0$ for
$q\in\partial X$ and $\|u-1\|_{1}\leq \frac{1}{2}$.  Note that this space is non-empty and closed
in the continuous functions with supremum norm.
\end{definition}
To define the operator $\Psi$  we need a little more notation. Let $S\subset V_{l_{1}}$ consist of
those points that lie in some $Y_{j}$ and in at least one other $l_{1}$-cell.  If $l_{1}$ is
sufficiently large then no two of the $Y_{j}$ can intersect; we assume this and see that the
connected components of $X\setminus S$ are the cells $Y_{j}$ (less points of $S$) and the set
$X\setminus Y$. Label those boundary points of the cell $Y_{j}$ at which $Y_{j}$ intersects another
$l_{1}$ cell by $x_{i,j}$ for $i=1,\dotsc, I_{j}$. Fixing a second scale $l_{2}$, also with size to
be determined, we associate to each $x_{i,j}$ the unique $(l_{1}+l_{2})$-cell in $Y_j$ containing
$x_{i,j}$, calling it $Z_{i,j}$.  We also set $Z_{0,j}=F_{j}^{l_{1}+l_{2}}(X)$, so it is the
$(l_{1}+l_{2})$-cell in $Y_{j}$ that contains $q_{j}\in V_{0}$, and define $w_{i,j}$ to be the word
such that $F_{i,j}(X)=F_{w_{i,j}}(X)=Z_{i,j}$. Figure \ref{FixedPoint-notationpic} illustrates our
labelling conventions in the case $X=SG$, $l_{1}=2$ and $l_{2}=1$.
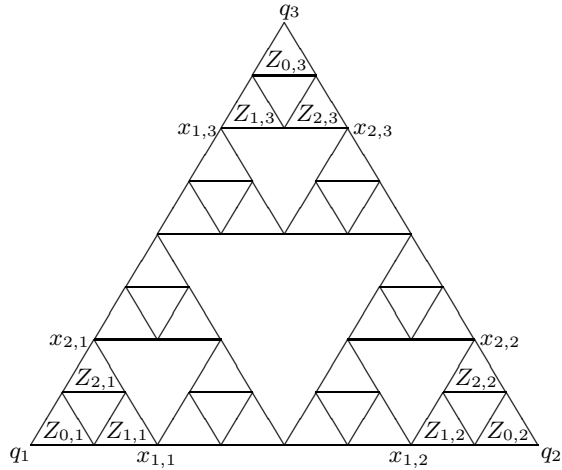
\begin{figure}[htb]
\centerline{
\begin{picture}(192,180)(0,0)
    \setlength{\unitlength}{.5pt}
    \put(0,10){\Spic{8}{8}{0}{0}}
    \small
    \put(-4,9){ \put(13,7){$Z_{0,1}$} \put(61,7){$Z_{1,1}$} \put(37,47){$Z_{2,1}$} \put(84,-12){$x_{1,1}$} \put(18,77){$x_{2,1}$} }%
    \put(358,9){ \put(-13,7){$Z_{0,2}$} \put(-61,7){$Z_{1,2}$} \put(-37,47){$Z_{2,2}$} \put(-87,-12){$x_{1,2}$} \put(-18,77){$x_{2,2}$} }%
    \put(138,249){ \put(15,7){$Z_{1,3}$} \put(63,7){$Z_{2,3}$} \put(39,47){$Z_{0,3}$} \put(-27,-3){$x_{1,3}$} \put(106,-3){$x_{2,3}$} }%
    \put(-16,0){$q_{1}$}%
    \put(386,0){$q_{2}$}%
    \put(188,336){$q_{3}$}%
\end{picture}}
\caption{Notation if $X=SG$, $l_{1}=2$ and $l_{2}=1$.}\label{FixedPoint-notationpic}
\end{figure}
We identify a particular function that is in $\mathcal{C}$ when $l_{1}$ is large enough.  Let $f$
be the piecewise harmonic function on $X$ with values $f(x_{i,j})=1$ for $i=1,\dotsc,I_{j}$ and
$j=1,\dotsc,N$ but $f(x_{0,j})=0$ for all $j$. It is clear that $f$ is continuous, identically $1$
on $X\setminus Y$ and harmonic on each of the sets $Y_{j}$. It fails to be harmonic only at the
points $x_{i,j}$ with $i\geq 1$, and we readily compute that the Laplacian of $f$ is a measure
supported at these points. In fact if $\delta_{x}$ denotes the Dirac mass at $x$ then
\begin{equation}\label{FixedPoint-laplacianoffunctionf}
    \Delta f
    = \sum_{j=1}^{N}\sum_{i=1}^{I_{j}} a_{i,j} \delta_{x_{i,j}}
    = -\sum_{j=1}^{N}\sum_{i=1}^{I_{j}} \partial_{n}f_{j}(x_{i,j}) \delta_{x_{i,j}}
    \end{equation}
with the second equality coming from the matching conditions for the Laplacian.  Since $f_{j}\circ
F_{j}^{l_{1}}$ is harmonic on $X$ with boundary values $0$ at $q_{j}$ and $1$ at all $q_{i}$,
$i\neq j$, it has normal derivatives bounded by a constant $C(r)$ that may be chosen so as to
depend only on the harmonic structure.  Scaling of the normal derivatives under $F_{j}^{l_{1}}$
therefore implies $|a_{i,j}|=|\partial_{n}f_{j}(x_{i,j})|\leq C(r) r_{j}^{-l_{1}}$.

The smooth bump function we seek will actually be a perturbation of $f$, constructed by iteratively
replacing the Dirac masses in \eqref{FixedPoint-laplacianoffunctionf} by rescaled copies of the
stage $k$ bump, correcting for the boundary normal derivatives, and applying the Dirichlet Green's
operator to obtain the stage $k+1$ bump. We will see that each stage gains one order of smoothness,
so the limiting function will be in $\dom(\Delta^{\infty})$. Our first step is to estimate the
effect of a perturbation of the type described.
\begin{lemma}\label{FixedPoint-changingmeasurelemma}
For each $j=1,\dotsc,N$ and $i=1,\dotsc,I_{j}$ let $\nu_{i,j}$ be a finite, signed Borel measure
with support in $Z_{i,j}$.  If we use the coefficients in \eqref{FixedPoint-laplacianoffunctionf}
to define
\begin{equation*}
    \nu = \sum_{j=1}^{N}\sum_{i=1}^{I_{j}} a_{i,j} \nu_{i,j}
    \end{equation*}
and let $u=G(\nu)$ be the result of applying the Dirichlet Green's operator, then
\begin{equation}\label{FixedPoint-estimatefromijtermsonallX}
    | u(x)| \leq  C(r) N^{2}  \sup_{i,j} \| \nu_{i,j} \| \quad\text{for all $x\in X$}
    \end{equation}
where $\|\nu_{i,j}\|$ is the total variation of $\nu_{i,j}$.  If in addition we have $\int
\nu_{i,j}=0$ for all $i$ and $j$ then
\begin{equation}\label{FixedPoint-estimatefromijtermsonXlessY}
    | u(x)| \leq  C(r) N \Bigl( \sup_{i,j}\|\nu_{i,j}\| \Bigr) \biggl( \sum_{k=1}^{N}
    r_{k}^{l_{2}} \biggr) \quad\text{for all $x\in X\setminus Y$}
    \end{equation}
\end{lemma}
\begin{proof}
Recall that $G$ may be represented as integration against the continuous kernel $-g(x,y)$, with
sign chosen so $g(x,y)\geq0$.  The estimates we desire follow from
\eqref{Introduction-Defnofgreensfn} and the fact that $|a_{i,j}|\leq C(r)r_{j}^{-l_{1}}$.  The
former ensures both that $|g(x,y)|\leq C(r) r_{j}^{l_{1}}$ on each $Y_{j}$ and that the oscillation
of $g(x,y)$ on $Z_{i,j}$ is at most $C(r)r_{w_{i,j}}$. We compute
\begin{align*}
    |u(x)|
    &= \biggl| \int_{X} g(x,y) d\nu(y) \biggr|\\
    &\leq \sum_{j=1}^{N} \sum_{i=1}^{I_{j}} |a_{i,j}| \int_{Z_{i,j}} |g(x,y)| d\nu_{i,j}(y) \\
    &\leq \sum_{j=1}^{N}\sum_{i=1}^{I_{j}} C(r) r_{j}^{-l_{1}} \int_{Z_{i,j}} C(r)r_{j}^{l_{1}} d|\nu_{i,j}(y)|\\
    &\leq C(r) N^{2} \sup_{i,j} \|\nu_{i,j}\|
    \end{align*}
which establishes the first inequality.  To obtain the second we observe that $\Delta u=0$ on
$X\setminus Y$, so by the maximum principle we need only verify the inequality at the points
$x_{i,j}$. Fix such a point $x_{i^{\prime},j^{\prime}}$, and use that each $\int d\nu_{i,j}=0$ to
subtract the appropriate constant from each integrand before estimating:
\begin{align*}
    |u(x_{i^{\prime},j^{\prime}})|
    &= \biggl| \int_{X} g(x_{i^{\prime},j^{\prime}},y) d\nu(y) \biggr|\\
    &= \biggl| \sum_{j=1}^{N} \sum_{i=1}^{I_{j}} a_{i,j} \int_{Z_{i,j}} g(x_{i^{\prime},j^{\prime}},y) d\nu_{i,j}(y)
    \biggr|\\
    &= \biggl| \sum_{j=1}^{N} \sum_{i=1}^{I_{j}} a_{i,j} \int_{Z_{i,j}}
    \Bigl( g(x_{i^{\prime},j^{\prime}},y)-g(x_{i^{\prime},j^{\prime}},x_{i,j}) \Bigr) d\nu_{i,j}(y)  \biggr|\\
    &\leq \sum_{j=1}^{N}\sum_{i=1}^{I_{j}} C(r) r_{j}^{-l_{1}} \int_{Z_{i,j}} C(r)r_{w_{i,j}} d|\nu_{i,j}|\\
    &\leq C(r) \sum_{j=1}^{N}\sum_{i=1}^{I_{j}} r_{k(i,j)}^{l_{2}} \|\nu_{i,j}\|
    \end{align*}
because $r_{w_{i,j}}=r_{j}^{l_{1}}r_{k}^{l_{2}}$ for some $k=k(i,j)$.  Each $k(i,j)$ occurs at most
once for a fixed $j$, so we obtain
\begin{equation*}
    |u(x_{i^{\prime},j^{\prime}})|
    \leq C(r) \Bigl( \sup_{i,j}\|\nu_{i,j}\| \Bigr) \biggl( \sum_{j=1}^{N}\sum_{k=1}^{N}
    r_{k}^{l_{2}} \biggr)
    \leq C(r) N \Bigl( \sup_{i,j}\|\nu_{i,j}\| \Bigr) \biggl( \sum_{k=1}^{N}
    r_{k}^{l_{2}} \biggr)
    \end{equation*}
\end{proof}

There is an analogous but simpler estimate for the effect of introducing a mass supported on one of
the small cells $Z_{0,j}$ at the boundary.  If $\nu_{0,j}$ is a finite, signed Borel measure with
support in $Z_{0,j}$ we use \eqref{Introduction-Defnofgreensfn} to see $|g(x,y)|\leq
C(r)r_{j}^{l_{1}+l_{2}}$ on $Z_{0,j}$ and therefore
\begin{equation}\label{FixedPoint-estimateforintroducingmassonboundarycell}
    \bigl| G(\nu_{0,j} ) \bigr|
    \leq \biggl| \int_{Z_{0,j}} |g(x,y)| d|\nu_{0,j}(y)| \leq C(r) r_{j}^{l_{1}+l_{2}} \|\nu_{0,j}\|
    \end{equation}

The ideas discussed so far allow us to generalize the definition of the operator $\Psi$ in
\eqref{FixedPoint-defnofoperatorPsiforI} and \eqref{FixedPoint-defnofoperatorPsiforSG}.  The idea
is that to replace the Dirac mass terms on the boundary cells $Z_{i,j}$ in
\eqref{FixedPoint-laplacianoffunctionf} with normalized rescaled copies of  $u\in\mathcal{C}$ and
apply the Green's operator to obtain a function that is near $1$ on $X\setminus Y$.  By adding some
terms on the cells $Z_{0,j}$ we can make the result have vanishing normal derivatives at the
boundary without changing the value on $X\setminus Y$ very much.  In consequence we will obtain an
operator that smooths  $u\in\mathcal{C}$ to be in $\mathcal{C}\cap\dom(\Delta)$ with vanishing
normal derivatives and is near $1$ on $X\setminus Y$.  Iterating the operator will then produce a
sequence of smoother and smoother bump functions.

Let
\begin{equation}\label{FixedPoint-defnofuij}
    u_{i,j} (x) = \begin{cases}
    \mu(Z_{i,j})^{-1} \Bigl( \int_{X} u d\mu \Bigr)^{-1} \Bigl( u\circ F_{i,j}^{-1} (x)\Bigr) &\text{if $x\in Z_{i,j}$}\\
    0 &\text{otherwise}
        \end{cases}
    \end{equation}
so that each $u_{i,j}$ is continuous and has integral $1$.  Since $u\in\mathcal{C}$ we also have
that $\int |u_{i,j}| \leq \Bigl(\int u \Bigr)^{-1}\|u\|\leq 3$.
\begin{definition}\label{FixedPoint-defnoffunctionPsidefn}
The operator $\Psi$ on $\mathcal{C}$ is $\Psi u=G(v)$, where
\begin{equation}\label{FixedPoint-defnoffunctionv}
    v(x)
    = \sum_{j=1}^{N} \sum_{i=0}^{I_j} b_{i,j} r_{j}^{-l_{1}} u_{i,j}(x)
    \end{equation}
and $G$ is the Dirichlet Green's operator. In this expression the coefficients for $i\geq 1$ are
given by $b_{i,j}= r_{j}^{l_{1}} a_{i,j}$ with $a_{i,j}$ as in
\eqref{FixedPoint-laplacianoffunctionf}, but the $b_{0,j}$ are yet to be determined.
\end{definition}
Note that $|b_{i,j}|\leq C(r)$ when $i\geq 1$. It is immediate that $G(v)\equiv 0$ on $\partial X$
and is continuous. Moreover $\Delta G(v)=v$ is a linear combination of continuous functions, so
$\Psi u\in\dom(\Delta)$.  The next lemma uses the Gauss-Green formula to reduce finding the correct
$b_{0,j}$ to a problem in linear algebra.
\begin{lemma}\label{FixedPoint-therearebzerojlemma}
If $l_{1}$ and $l_{2}$ are sufficiently large then there are values $b_{0,j}$ such that
$\partial_{n}\Psi u\equiv 0$ on $\partial X$.  The minimal sizes of $l_{1}$ and $l_{2}$ depend only
on the harmonic structure of $X$ and the number of vertices $N$ in $\partial X$.
\end{lemma}
\begin{proof}
Let $h_{j}$ be the function that is harmonic on $X$, equal to $1$ at $q_{j}$ and $0$ at all other
points of $\partial X$. Using $\Delta G(v)=v$, $G(v)\equiv 0$ on $\partial X$, and the Gauss-Green
formula
\begin{align*}
    \int h_{j}(y)v(y)d\mu(y)
    &= \int h_{j}(y)\Delta G(v)(y) d\mu(y)\\
    &= \sum_{q_{k}\in \partial X}  h_{j}(q_{k}) \bigl(\partial_{n}G(v)\bigr)(q_{k})\\
    &= \bigl(\partial_{n}G(v)\bigr)(q_{j})
    \end{align*}
from which $\partial_{n}G(v)\equiv0$ on $\partial X$ is simply
\begin{equation*}
    0 = \int h_{j}(y)v(y)d\mu(y)
    = \sum_{i^{\prime},j^{\prime}} b_{i^{\prime},j^{\prime}} r_{j^{\prime}}^{-l_{1}} \int h_{j}(y) u_{i^{\prime},j^{\prime}}(y)  d\mu(y)
    \end{equation*}
for all $j=1,\dotsc,N$.  Moving the terms depending on the fixed values $b_{i^{\prime},j^{\prime}}$
for $i^{\prime}\geq 1$ this may be reformulated as
\begin{align}
    \sum_{j^{\prime}} b_{0,j^{\prime}} r_{j^{\prime}}^{-l_{1}} \int h_{j}(y) u_{0,j^{\prime}}(y)  d\mu(y)
    &= - \sum_{j^{\prime}=1}^{N}\sum_{i^{\prime}=1}^{I_{j^{\prime}}} b_{i^{\prime},j^{\prime}}r_{j^{\prime}}^{-l_{1}}
    \int h_{j}(y) u_{i^{\prime},j^{\prime}}(y) d\mu(y)  \label{FixedPoint-secondconditoncoeffsforv}
    \end{align}
which we recognize as a matrix equation $\sum_{j^{\prime}} M(u)_{j^{\prime},j}\,
r_{j^{\prime}}^{-l_{1}} b_{0,j^{\prime}} = A(u)_{j}$ with
\begin{gather}
    \bigl( M(ud\mu) \bigr)_{j^{\prime},j}
    = \int h_{j}(y) u_{0,j^{\prime}}(y) d\mu(y) \label{FixedPoint-definitionofM}\\
    (A(ud\mu))_{j}
    = - \sum_{j^{\prime}=1}^{N}\sum_{i^{\prime}=1}^{I_{j^{\prime}}} b_{i^{\prime},j^{\prime}}r_{j^{\prime}}^{-l_{1}}
    \int h_{j}(y) u_{i^{\prime},j^{\prime}}(y) d\mu(y) \label{FixedPoint-definitionofA}
    \end{gather}
It is clear that we need to know $M=M(ud\mu)$ is invertible, but rather than prove this directly we
do so by proving a perturbation estimate similar to Lemma \ref{FixedPoint-changingmeasurelemma}
that will be useful later.  To this end consider replacing each of the measures
$u_{0,j^{\prime}}d\mu$ in \eqref{FixedPoint-definitionofM}  with a copy of a different probability
measure $d\sigma$ scaled and translated to give $d\sigma_{i^{\prime},j^{\prime}}$ supported on
$Z_{i^{\prime},j^{\prime}}$. We call the result $M(d\sigma)$. The difference of these measures has
mass zero, so we can compute an estimate involving the total variation of the measures
\begin{align}
    \Bigl| M(ud\mu-d\sigma)_{j^{\prime},j} \Bigr|
    &= \biggl| \int h_{j}(y) \bigl( u_{0,j^{\prime}}(y) d\mu(y) - d\sigma_{0,j^{\prime}}(y) \bigr)
    \biggr| \notag\\
    &\leq \biggl| \int \bigl( h_{j}(y)- h_{j}(x_{0,j^{\prime}}) \bigr)
        \Bigl( u_{0,j^{\prime}}(y) d\mu(y) - d\sigma_{0,j^{\prime}}(y) \Bigr) \notag \\
    &\leq C(r) r_{j^{\prime}}^{l_{1}+l_{2}} \|ud\mu -d\sigma\| \label{FixedPoint-estimateforM}
    \end{align}
because $h_{j}$ is harmonic and therefore varies by at most $C(r)r_{j^{\prime}}^{l_{1}+l_{2}}$ on
each $Z_{0,j^{\prime}}$.  In particular if the measures $d\sigma_{i^{\prime},j^{\prime}}$ are Dirac
masses at the points $x_{0,j^{\prime}}$ then $M(d\sigma)$ is simply the identity, so
\eqref{FixedPoint-estimateforM} implies
\begin{equation*}
    \bigl|(I-M)_{j^{\prime},j} \bigr|
    \leq C(r) r_{j^{\prime}}^{l_{1}+l_{2}}
    \end{equation*}
from which $M$ is invertible when $l_{1}+l_{2}$ is large, with $\|I- M^{-1}\|\leq
C(N,r)\sum_{j}r_{j}^{l_{1}+l_{2}}$.

A similar perturbation argument can be made for $A(d\mu-d\sigma)$, where $A(d\sigma)$ is obtained
by replacing each  $u_{i^{\prime},j^{\prime}}d\mu$ by $d\sigma_{i^{\prime},j^{\prime}}$ in
\eqref{FixedPoint-definitionofA}. Estimating the integral terms and using the bound
$|b_{i^{\prime},j^{\prime}}|\leq C(r)$ we obtain
\begin{equation}
    \bigl| A(ud\mu-d\sigma)_{j} \bigr|
    \leq C(N,r) \, \Bigl(\sum_{i} r_{i}^{l_{2}}\Bigr) \, \| u d\mu -d\sigma\|
    \label{FixedPoint-estimateofA}
    \end{equation}
however this is not the most useful thing we can do here.  Instead we recognize that the bounds
$|h_{j}(y)- 1|\leq C(r)r_{j}^{l_{1}}$ on $Y_{j}$ and $|h_{j}(y)|\leq r_{j^{\prime}}^{l_{1}}$ on
$Y_{j^{\prime}}$ for $j^{\prime}\neq j$\, ensure
\begin{equation*}
    \biggl| A(ud\mu)_{j} + \sum_{i^{\prime}=1}^{I_{j}} b_{i^{\prime},j} r_{j}^{-l_{1}} \biggr|
    \leq C(N,r)
    \end{equation*}
so that combining this with our bound on $I-M^{-1}$ we have
\begin{equation}
    \biggl| b_{0,j} -  \sum_{i^{\prime}=1}^{I_{j}} b_{i^{\prime},j} \biggr|
    \leq C(N,r)r_{j}^{l_{1}}
    \end{equation}

If we examine the function $f$ in \eqref{FixedPoint-laplacianoffunctionf} it is clear that the
normal derivative at each point $x_{0,j}$ is $\sum_{i=1}^{I_{j}} b_{i,j}$, so our choice of
$b_{0,j}$ is a small perturbation of that which would be used to cancel the normal derivatives of
$f$.  We also remark that this shows all $|b_{0,j}|\leq C(N,r)$.
\end{proof}

If $l_{1}$ and $l_{2}$ are large enough then the values $b_{0,j}$ from Lemma
\ref{FixedPoint-therearebzerojlemma} may be used to complete Definition
\ref{FixedPoint-defnoffunctionPsidefn} for the operator $\Psi$.  Some key properties of this
operator are summarized in the following lemma.
\begin{lemma}\label{FixedPoint-mainlemma}
    If $l_{1}$ and $l_{2}$ are sufficiently large then $\Psi(u)\in\mathcal{C}\cap\dom(\Delta)$ and
    \begin{align}
        \bigl\| \Psi u \bigr\|_{\infty} &\leq C_{1}  \label{FixedPoint-desiredLinftybdforPsiu}\\
        \bigl| \Psi u(y) - 1 \bigr| &\leq C_{2}  \sum_{j=1}^{N} r_{j}^{l_{2}} \quad \text{for all $y\in X\setminus Y$}
        \label{FixedPoint-desiredestimateonXminusY}
        \end{align}
    where $C_{1}$, $C_{2}$ and the minimal sizes of $l_{1}$ and $l_{2}$ are constants depending only on the harmonic structure of $X$,
    the measure $\mu$ and the number of vertices $N$ in $\partial X$.
    \end{lemma}
\begin{proof}
Since
\begin{equation*}
    \Psi (u)
    = G \biggl( \sum_{j=1}^{N} \sum_{i=1}^{I_{j}} a_{i,j} u_{i,j}  + \sum_{j=1}^{N}
    b_{0,j}r_{j}^{-l_{1}} u_{0,j} \biggr)
    \end{equation*}
we obtain \eqref{FixedPoint-desiredLinftybdforPsiu} from
\eqref{FixedPoint-estimatefromijtermsonallX} and
\eqref{FixedPoint-estimateforintroducingmassonboundarycell}, and the fact that $|b_{0,j}|\leq
C(N,r)$ for all $j$.  The estimate \eqref{FixedPoint-desiredestimateonXminusY} is only a little
more difficult.  Using $f(x)= 1$ and \eqref{FixedPoint-laplacianoffunctionf} on the set $X\setminus
Y$ we see that
\begin{align}
    \bigl| \Psi u (x) - 1 \bigr|
    &=\bigl| \Psi u (x) - f(x) \bigr| \notag\\
    &=\Biggl| G \biggl( \sum_{j=1}^{N} \sum_{i=0}^{I_{j}} b_{i,j}r_{j}^{-l_{1}} u_{i,j} \biggr)
        - G\biggl(\sum_{j=1}^{N} \sum_{i=1}^{I_{j}} a_{i,j} \delta_{x_{i,j}} \biggr) \Biggr| \notag\\
    &\leq \Biggl| G \biggl( \sum_{j=1}^{N} \sum_{i=1}^{I_{j}} a_{i,j} \Bigl( u_{i,j}d\mu
    -\delta_{x_{i,j}} \Bigr) \biggr) \Biggr|
        + \sum_{j=1}^{N} \bigl| b_{0,j} r_{j}^{-l_{1}} G(u_{0,j}d\mu) \bigr| \notag\\
    &\leq C(N,r) \sup_{i,j} \bigl\| u_{i,j}d\mu -\delta_{x_{i,j}} \bigr\| \biggl( \sum_{k=1}^{N}
    r_{k}^{l_{2}} \biggr)
        + C(N,r) \biggl( \sum_{j=1}^{N} r_{j}^{l_{2}} \biggr) \|u_{0,j}d\mu \| \notag\\
    &\leq C(N,r)  \biggl( \sum_{j=1}^{N} r_{j}^{l_{2}} \biggr)
    \label{FixedPoint-cancellationcomputation}
    \end{align}
where the estimate for the $b_{0,j}$ terms came from
\eqref{FixedPoint-estimateforintroducingmassonboundarycell} and that for the $a_{i,j}$ terms is
from \eqref{FixedPoint-estimatefromijtermsonXlessY} because $\int u_{i,j}d\mu = 1 =
\int\delta_{x_{i,j}}$ and both are supported on $Z_{i,j}$.

Finally we check that $\|\Psi u-1\|_{1}\leq \frac{1}{2}$.  Using the results we have so far
\begin{align}
    \|\Psi u -1 \|_{1}
    &\leq \int_{X\setminus Y} C(N,r)  \biggl( \sum_{j=1}^{N} r_{j}^{l_{2}} \biggr) d\mu  + \int_{Y}
    \bigl( 1+C_ {1} \bigr) d\mu \notag\\
    &\leq C(N,r)  \biggl( \sum_{j=1}^{N} r_{j}^{l_{2}} \biggr) + (1+C_{1})\mu(Y) \
    \end{align}
so that we can be sure $\Psi u\in\mathcal{C}$ if both $l_{1}$ and $l_{2}$ are sufficiently large,
because $\mu(Y)\rightarrow0$ as $l_{1}\rightarrow\infty$.  It has already been observed that
$u\in\dom(\Delta)$ and $u\equiv0$ on $\partial X$, so the lemma is proven.
\end{proof}

Finally we come to the main result of this section.  The following theorem implements the idea that
motivated our definition of $\Psi$, namely that $\Psi$ smoothes functions in $\mathcal{C}$ and
therefore its recursive application gives a bump function in $\dom(\Delta^{\infty})$.
\begin{theorem}\label{FixedPoint-smoothbumpexiststheorem}
    Given $\epsilon>0$ there are $l_{1}$ and $l_{2}$ sufficiently large that $\Psi$
    has a fixed point $u_{0}$ in $\mathcal{C}$ with $|u-1|\leq \epsilon$ on $X\setminus Y$.
    The fixed point is a smooth bump function and every $u\in\mathcal{C}$ has
    $\|\Psi^{k} u- u_{0}\|_{\infty}\rightarrow0$ as $k\rightarrow\infty$.
    \end{theorem}
\begin{proof}
Let  $u,\tilde{u}\in\mathcal{C}$. We calculate $\Psi u - \Psi \tilde{u}= G (v - \tilde{v})$, where
$v$ and $\tilde{v}$ are as in \eqref{FixedPoint-defnoffunctionv}.  Beginning with a variant of the
computation \eqref{FixedPoint-cancellationcomputation} we have
\begin{align}
    \bigl| G(v-\tilde{v} ) (x)\bigr|
    &\leq \Biggl| G \biggl( \sum_{j=1}^{N} \sum_{i=1}^{I_{j}} a_{i,j} \bigl(
    u_{i,j}-\tilde{u}_{i,j} \bigr) \biggr) \Biggr|
        + \biggl| G \Bigl( \sum_{j=1}^{N} r_{j}^{-l_{1}} \bigl( b_{0,j} u_{0,j} -\tilde{b}_{0,j} \tilde{u}_{0,j} \bigr)
        \Bigr) \Biggr| \notag\\
    &\leq \Biggl| G \biggl( \sum_{j=1}^{N} \sum_{i=1}^{I_{j}} a_{i,j} \bigl(
    u_{i,j}-\tilde{u}_{i,j} \bigr) \biggr) \Biggr|
        + \sum_{j=1}^{N} r_{j}^{-l_{1}} |\tilde{b}_{0,j}| \bigl| G(u_{0,j}-\tilde{u}_{0,j} ) \bigr| \notag\\
    &\qquad  + \sum_{j=1}^{N} r_{j}^{-l_{1}} \bigl| b_{0,j} -\tilde{b}_{0,j} \bigr| |G(u_{0,j})(x)|
    \label{FixedPoint-perturbationestforuandtildeu}
    \end{align}
which suggests we will need to know estimates for both $(u_{i,j}-\tilde{u}_{i,j})$ and
$|b_{0,j}-\tilde{b}_{0,j}|$.  Conveniently we can reduce the latter to the former using
\eqref{FixedPoint-estimateforM} and \eqref{FixedPoint-estimateofA}, because $b_{0,j}$ and
$\tilde{b}_{0,j}$ are computed from equations of the form $\sum_{j}M(ud\mu)_{j,i} b_{0,j}
r_{j}^{-l_{1}} = A(ud\mu)_{i}$.  We easily see that
\begin{equation*}
   r_{j}^{-l_{1}} \bigl(b_{0,j}-\tilde{b}_{0,j} \bigr)
    = \Bigl( M(ud\mu)^{-1} A(ud\mu-\tilde{u}d\mu) \Bigr)
    + \Bigl( M(ud\mu)^{-1} M(\tilde{u}d\mu -ud\mu) \bigl( r_{j}^{-l_{1}}b_{0,j} \bigr) \Bigr)
    \end{equation*}
however by \eqref{FixedPoint-estimateforM} we have both $\|M(ud\mu-\tilde{u}d\mu) \|\leq
C(N,r)\|u-\tilde{u}\|_{1}\sum_{i}r_{i}^{l_{1}+l_{2}}$ and that $\|I-M^{-1}(ud\mu)\|\leq
C(N,r)\|u\|_{1}\sum_{i}r_{i}^{l_{1}+l_{2}}$, while \eqref{FixedPoint-estimateofA} gives us that
$\|A(ud\mu-\tilde{u}d\mu)\|\leq C(N,r) \sum_{i}r_{i}^{l_{2}} \|u-\tilde{u} \|_{1}$.  In both cases
we have used that the total variation of $u_{i,j}-\tilde{u}_{i,j}$ is bounded by
$\|u_{i,j}-\tilde{u}_{i,j}\|_{1}$ and that writing $u_{X}= \int_{X} u$ we can calculate
\begin{equation*}
    \int |u_{i,j}-\tilde{u}_{i,j}|
    = \int_{X} \Bigl| u_{X}^{-1} u(x) - \tilde{u}_{X} \tilde{u}(x)
    \Bigr|
    \leq \tilde{u}_{X}^{-1} \bigl(1+ |u|_{X} u_{X}^{-1} \bigr)  \int |u-\tilde{u}|
    \leq 8 \|u-\tilde{u}\|_{1}
    \end{equation*}
The conclusion is then that $r_{j}^{-l_{1}}|b_{0,j}-\tilde{b}_{0,j}|\leq
C(N,r)\|u-\tilde{u}\|_{1}\sum_{i}r_{i}^{l_{2}}$.   Substituting this into
\eqref{FixedPoint-perturbationestforuandtildeu} and using
\eqref{FixedPoint-estimatefromijtermsonXlessY} and
\eqref{FixedPoint-estimateforintroducingmassonboundarycell} we find that on $X\setminus Y$
\begin{align*}
    \bigl| G(v-\tilde{v} ) (x)\bigr|
    &\leq C(N,r) \biggl( \sum_{i=1}^{N} r_{i}^{l_{2}} \biggr) \sup_{i,j} \bigl\|u_{i,j}d\mu-\tilde{u}_{i,j}d\mu \bigr\|
        + C(N,r) \Bigl( \sum_{i}r_{i}^{l_{2}} \Bigr) \|u-\tilde{u}\|_{1} \\
    &\leq C(N,r) \Bigl( \sum_{i}r_{i}^{l_{2}} \Bigr) \|u-\tilde{u}\|_{1}
    \end{align*}
because the total variation $\|u_{i,j}d\mu-\tilde{u}_{i,j}d\mu\|$ was already computed to be at
most $8\|u-\tilde{u}\|_{1}$.  On the rest of $X$ we  must use
\eqref{FixedPoint-estimatefromijtermsonallX} instead of
\eqref{FixedPoint-estimatefromijtermsonXlessY}.  The weaker estimate is easily computed to be
\begin{equation} \label{FixedPoint-LoneboundforPsi}
    \bigl| G(v-\tilde{v} ) (x)\bigr|
    \leq C(N,r) \|u-\tilde{u}\|_{1}
    \end{equation}

From our estimates on $G(v-\tilde{v}) = \Psi u-\Psi \tilde{u}$ we see that $\Psi$ is a contraction
on $L^{1}$ if $l_{1}$ and $l_{2}$ are sufficiently large, because
\begin{align*}
    \bigl\| \Psi u - \Psi\tilde{u} \bigr\|_{1}
    &\leq  \mu(X\setminus Y) C(N,r) \Bigl( \sum_{i}r_{i}^{l_{2}} \Bigr) \|u-\tilde{u}\|_{1}
    + \mu(Y) C(N,r) \|u-\tilde{u}\|_{1} \\
    &\leq C(N,r) \Bigl( \mu(Y) + \sum_{j} r_{j}^{l_{2}} \Bigr) \bigl\| u  - \tilde{u} \bigr\|_{1}
    \end{align*}
It follows readily that $\Psi$ has a unique fixed point in $\mathcal{C}$  and $\Psi^{k}u$ converges
to this fixed point in $L^{1}$.  From \eqref{FixedPoint-LoneboundforPsi} this convergence is
uniform, and we notice that the correct choice of $l_{2}$ provides $|u_{0}-1|=|\Psi u_{0}-1|\leq
\epsilon$ on $X\setminus Y$ by \eqref{FixedPoint-desiredestimateonXminusY}.

It remains only to see that $u_{0}$ is a smooth bump function on $X$.~ Inductively suppose
$\Psi^{k}u\in\dom(\Delta^{k})$ and both $\Delta^{j}\Psi^{k}u\equiv 0$ on $\partial X$ for $0\leq
j\leq k$ and $\partial_{n}\Delta^{j}\Psi^{k}u\equiv 0$ on $\partial X$ for $0\leq j\leq k-1$.  This
is certainly true for $k=0$. By construction, $\Delta \Psi^{k+1}u$ is a linear combination of
rescaled copies of $\Psi^{k}u$ that have been extended by zero as in \eqref{FixedPoint-defnofuij}.
Each of these functions is in $\dom(\Delta^{k})$ by the matching conditions for the Laplacian, so
we conclude that $\Psi^{k+1}u\in\dom(\Delta^{k+1})$.  It is immediate that
$\Delta^{j}\Psi^{k+1}u\equiv 0$ on $\partial X$ for $1\leq j\leq k+1$ and
$\partial_{n}\Delta^{j}\Psi^{k+1}u\equiv 0$ on $\partial X$ for $1\leq j\leq k$.  By Lemma
\ref{FixedPoint-mainlemma} we know also that $\Psi^{k+1}u$ and $\partial_{n}\Psi^{k+1}u$ vanish on
$\partial X$, which closes the induction and establishes that $u_{0}\in \dom(\Delta^{\infty})$ and
vanishes to infinite order on $\partial X$.
\end{proof}


\section{A Borel theorem on p.c.f. fractals}\label{Borel_section}

The classical Borel theorem tells us that given any neighborhood of $x_{0}\in\mathbb{R}$ and any
prescribed sequence of values for $u$ and its derivatives at $x_{0}$, we may construct a smooth
function $u$ with support in the neighborhood and the given sequence of derivatives at $x_{0}$.
Using the smooth bump functions we have constructed, we now show that the same result holds at
junction points of certain p.c.f. fractals.  In what follows $X$ is p.c.f. and self-similar under
$\{F_{j}\}_{j=1}^{N}$ and the measure $\mu$ is self-similar with factors $0<\mu_{j}<1$,
$\sum_{1}^{N} \mu_{j}=1$, so that $\mu(F_{w}(X))=\prod_{j=1}^{m} \mu_{w_{j}}$ when $w$ is the word
$w_{1}\dotsc w_{m}$.  The Dirichlet form is that associated to a regular self-similar harmonic
structure with resistance renormalization factors $0<r_{j}<1$ for $j=1,\dotsc,N$.  Our arguments
depend on the existence of smooth bumps as previously constructed.  The crucial step is the
existence of smooth functions with finitely many prescribed normal derivative values, which is
established in the following lemma.
\begin{lemma}\label{Borel-existgencepffunctionsfl}
    Given a boundary point $q\in V_{0}$ there are smooth functions $f_{l}$ such that
    \begin{gather*}
        \Delta^{k}f_{l}(p) = 0 \quad \text{for all $k\geq 0$ and all $p\in V_{0}$}\\
        \partial_{n}\Delta^{k}f_{l}(p) = 0 \quad\text{for all $k\geq 0$ and all $p\in
        V_{0}\setminus\{q\}$}\\
        \partial_{n} \Delta^{k} f_{l} (q) = \delta_{lk}
        \end{gather*}
    \end{lemma}
\begin{proof}
We begin with the case $l=0$. If $U$ is the smooth bump function on $X$ produced in Theorem
\ref{FixedPoint-smoothbumpexiststheorem} we localize it near the boundary points of $X$ at a scale
$m$ to be determined later. Define
\begin{equation*}
    U_{j} = \begin{cases}
            U\circ F_{j}^{-m} &\text{on $F_{j}^{m}(X)$}\\
            0 &\text{otherwise}
        \end{cases}
    \end{equation*}
and observe from the matching conditions for the Laplacian that each $U_{j}$ is smooth.  Now apply
the Dirichlet Green's operator $G$ to these functions and form the linear combination
\begin{equation*}
    f = \sum_{j=1}^{N} a_{j}G(U_{j})
    \end{equation*}
with coefficients to be chosen.  It is clear from the properties of $U$ that $\Delta^{k}f=0$ on
$V_{0}$ for all $k\geq 0$ and that $\partial_{n}\Delta^{k} f=0$ on $V_{0}$ if $k\geq 1$. Moreover
the Gauss-Green formula yields values of the normal derivatives at the points $q_{i}\in V_{0}$
\begin{equation*}
    \partial_{n} f(q_{i}) = - \sum_{j=1}^{N} a_{j} \int_{X} h_{i} U_{j}
    \end{equation*}
where $h_{i}$ is the harmonic function on $X$ with boundary values $h_{i}(q_{j})=\delta_{ij}$. In
order that there be coefficients $a_{j}$ such that $f$ has the properties asserted for $f_{0}$ it
then suffices that we can invert the matrix with entries $A_{ij}=\int h_{i}U_{j}$.  We use the fact
that
\begin{gather*}
    |h_{i}|\leq r_{j}^{m} \quad\text{on} \quad F_{j}^{  m}(X) \quad\text{for}\quad j\neq i\\
    |h_{i}-1|\leq r_{i}^{m} \quad\text{on}\quad F_{i}^{  m}(X)
    \end{gather*}
which follow from the estimates on the oscillation of a harmonic function that were mentioned at
the end of the introduction.  Using this we calculate
\begin{gather}
    |A_{ij}| = \left| \int h_{i} U_{j} \right| \leq r_{j}^{m} \mu_{j}^{m} \int_{X} |U| \quad\text{for
    $j\neq i$}\notag\\
    \left| A_{ii}-\mu_{i}^{m}\int_{X} U \right| = \left| \int (h_{i}-1)U_{i} \right| \leq
    r_{i}^{m}\mu_{i}^{m} \int_{X} |U| \label{Borel-matrixneardiagonal}
    \end{gather}
Let $D$ be the diagonal matrix with entries $D_{ii}=\mu_{i}^{m}\int_{X}U$.  Then we readily compute
$(AD^{-1})_{ij}=\mu_{j}^{-m}\bigl(\int U\bigr)^{-1}A_{ij}$ is close to the identity if $m$ is
large. Indeed, by \eqref{Borel-matrixneardiagonal} we have $|(I-AD^{-1})_{ij}|\leq C\rho^{m}$ with
$\rho=\max_{i}r_{i}$ and $C=\bigl(\int U\bigr)^{-1}\bigl(\int |U|\bigr)$, so that $AD^{-1}$ is
invertible provided $m$ is sufficiently large.

We proceed by induction on $l$, with an almost unchanged argument. Suppose the functions $f_{l}$
for $l\leq L-1$ have been constructed as linear combinations of the form
\begin{equation}\label{Borel-constructionoffl}
    f_{l} = \sum_{n=1}^{l}\sum_{j=1}^{N} a_{jn} G^{  (n+1)}(U_{j}) \quad l\leq L-1
    \end{equation}
and consider the function
\begin{equation*}
    f = \sum_{j=1}^{N} a_{j} G^{  (L+1)}(U_{j})
    \end{equation*}
so that $\Delta^{k}f=0$ on $V_{0}$ for all $k$ and $\partial_{n}\Delta^{k} f=0$ on $V_{0}$ for
$k\geq L+1$.  When $k=L$ we have
\begin{equation*}
    \partial_{n} \Delta^{L} f (q_{i}) = -\sum_{j=1}^{N} a_{j} A_{ij}
    \end{equation*}
where $A_{ij}$ is as before, so we may select $a_{j}$ to obtain $\partial_{n} \Delta^{L} f(q)=1$
and $\partial_{n} \Delta^{L} f(p)=0$ at other points $p\in V_{0}$. Subtracting an appropriate
linear combination of the $f_{l}$ for $l\leq L-1$ we obtain the desired $f_{L}$ in the form of
\eqref{Borel-constructionoffl}.
\end{proof}

With this in hand it is simple to deal with finitely many values of the Laplacian at a boundary
point $q\in V_{0}$.
\begin{lemma}
    Given a boundary point $q\in V_{0}$ there are smooth functions $g_{l}$ such that
    \begin{gather*}
        \partial_{n}\Delta^{k}g_{l}(p) = 0 \quad\text{for all $k\geq 0$ and all $p\in
        V_{0}$}\\
        \Delta^{k}g_{l}(p) = 0 \quad \text{for all $k\geq 0$ and all $p\in V_{0}\setminus \{q\}$}\\
        \Delta^{k}g_{l}(q) = \delta_{lk}
        \end{gather*}
    \end{lemma}
\begin{proof}
Let $h$ be the harmonic function which is $1$ at $q$ and $0$ at all other points of $V_{0}$.
Clearly $\Delta^{k}h\equiv 0$ for all $k\geq 1$ and therefore also $\partial_{n}\Delta^{k}h(p)=0$
for all $k\geq 1$ and $p\in V_{0}$. For each $p\in V_{0}$ let $f_{0,p}$ be the function constructed
in Lemma \ref{Borel-existgencepffunctionsfl} with non-vanishing normal derivative at $p$.  It is
clear that $g_{0}=h-\sum_{p\in V_{0}}\bigl(\partial_{n}h(p)\bigr)f_{0,p}$ has the desired
properties, so we have found the first of our functions. To obtain the others we simply apply the
Dirichlet Green's operator $G$.~ Notice that $\Delta^{k}G^{l}h(p)=\delta_{k,l}\delta_{p,q}$ for all
$k$ and all $p\in V_{0}$, and also $\partial_{n}\Delta^{k}G^{l}h(p)=0$ for all $k\geq l+1$.  To
obtain $g_{k}$ it remains only to subtract off all normal derivatives that occur for $0\leq k\leq
l$ using the functions from Lemma \ref{Borel-existgencepffunctionsfl}.
\end{proof}

The proof of a Borel-type theorem from the above lemmas is standard.  All that is needed is
information about how scaling the support of a function changes its Laplacian and normal
derivatives.  Recall that for a Dirichlet form associated to a regular self-similar resistance,
both the Laplacian and the normal derivative may be obtained as renormalized limits of
corresponding quantities defined on the approximating graphs (Section 3.7 of \cite{Kigamibook}). In
particular, pre-composition with the map $F_{i}^{-1}$ rescales the $k$-th power of the Laplacian by
$(\mu_{i}r_{i})^{-k}$ and its normal derivative by $\mu_{i}^{-k}r_{i}^{-k-1}$. For this reason, if
$q=q_{i}$ is the boundary point of interest we define
\begin{align*}
    f_{l,m} &= \begin{cases}
                \mu_{i}^{ml}r_{i}^{m(l+1)}  f_{l}\circ F_{i}^{  (-m)} &\text{on $F_{i}^{  m}$}\\
                0 &\text{otherwise}
                \end{cases}\\
    g_{l,m} &=\begin{cases}
                (\mu_{i}r_{i})^{ml} g_{l}\circ F_{i}^{  (-m)} &\text{on $F_{i}^{  m}$}\\
                0 &\text{otherwise}
                \end{cases}
    \end{align*}
so that we have for all $k$
\begin{align}
    \Delta^{k} f_{l,m} (q) = 0   & & \Delta^{k} g_{l,m}(q) = \delta_{lk} \notag\\
    \partial_{n} \Delta^{k} f_{l,m}(q) = \delta_{lk} & & \partial_{n} \Delta^{k} g_{l,m}(q) = 0
    \label{Borel-propertiesofflmandglm}
    \end{align}
but the $L^{\infty}$ norms of the lower order derivatives have decreased and those of the higher
order derivatives have increased.
\begin{gather}
    \|\Delta^{k}f_{l,m}\|_{\infty} = \mu_{i}^{m(l-k)}r_{i}^{m(l+1-k)} \|\Delta^{k}f_{l} \|_{\infty} \label{Borel-scalingestforflm}\\
    \|\Delta^{k}g_{l,m}\|_{\infty} = (\mu_{i} r_{i})^{m(l-k)} \|\Delta^{k}g_{l}\|_{\infty}
    \label{Borel-scalingestforglm}
    \end{gather}
With this in hand we can easily prove our version of the Borel theorem.
\begin{theorem}\label{Borel-borelthm}
    Let $q\in V_{0}$ be fixed, and\/ $\Omega$ be an open neighborhood of $q$.
    Given a jet $\rho=(\rho_{0},\rho_{1},\dotsc)$ of values for powers of the Laplacian and
    $\sigma=(\sigma_{0},\sigma_{1},\dotsc)$ of values for their normal derivatives, there is a
    smooth function $f$ with support in $\Omega$ and both $\Delta^{k}f(q)=\rho_{k}$ and
    $\partial_{n}\Delta^{k} f(q)=\sigma_{k}$ for all $k$.
    \end{theorem}
\begin{proof}
We give the usual proof that it is possible to define $f$\/ by the series
\begin{equation}\label{Borel-defnofborelfnasseries}
    f = \sum_{l} \biggl( \rho_{l}g_{l,m_{l}} + \sigma_{l} f_{l,n_{l}} \biggr)
    \end{equation}
for an appropriate choice of  $m_{l}$ and $n_{l}$.

Let $m_{0}=n_{0}$ be sufficiently large that $F_{i}^{m_{0}}(X)\subset\Omega$.  For each $l$ choose
$m_{l}\geq m_{0}$ so large that
\begin{equation*}
    \bigl\|\rho_{l} \Delta^{k} g_{l,m_{l}} \bigr\|_{\infty} \leq 2^{k-l-1} \quad\text{for
    $0\leq k\leq l-1$}
    \end{equation*}
using the scaling estimate \eqref{Borel-scalingestforglm}.  Similarly use the scaling relation
\eqref{Borel-scalingestforflm} to take $n_{l}\geq m_{0}$ such that
\begin{equation*}
    \bigl\|\sigma_{l} \Delta^{k} f_{l,n_{l}} \bigr\|_{\infty} \leq 2^{k-l-1} \quad\text{for
    $0\leq k\leq l-1$}
    \end{equation*}
Then for fixed $k$ we have
\begin{equation*}
    \left\| \sum_{l=k+1}^{\infty} \Delta^{k} \Bigl( \rho_{l}g_{l,m_{l}} + \sigma_{l} f_{l,n_{l}}
    \Bigr) \right\|_{\infty} \leq \sum_{k+1}^{\infty} 2^{k-l} \leq 1
     \end{equation*}
so that all powers of the Laplacian applied to \eqref{Borel-defnofborelfnasseries} produce
$L^{\infty}$ convergent series.  It follows that $f$ as defined in
\eqref{Borel-defnofborelfnasseries} is smooth and has support in $\Omega$.  By
\eqref{Borel-propertiesofflmandglm} it has the desired jet, so the result follows.
\end{proof}

We remark that for any $\epsilon>0$ we could replace the bounds $2^{k-l-1}$ in the proof with
$\epsilon 2^{k-l-1}$. It follows that we can define $f$ by \eqref{Borel-defnofborelfnasseries} and
have the estimate
\begin{equation}\label{Borel-usefulestimatefordistributiontheory}
    \bigl\| \Delta^{k} f \bigr\|_{\infty} \leq C(k,\Omega) \sum_{l=0}^{k} \bigl( |\rho_{l}|+|\sigma_{l}| \bigr)
    +\epsilon
    \end{equation}
where $C(k,\Omega)$ does not depend on the jet we prescribe.

It is also possible to estimate the effect of the size of the support $F_{i}^{m_{0}}(X)$ on a
finite collection of jet terms. Take $m_{l}=m_{0}$ for $0\leq j\leq L$, $n_{l}=m_{0}$ for $0\leq
j\leq L-1$ and thereafter choose both so large that the estimates $\epsilon 2^{k-l-1}$ hold. From
\eqref{Borel-scalingestforflm} and \eqref{Borel-scalingestforglm} we deduce for $0\leq k\leq L$,
\begin{equation}\label{Borel-usefulestimatefordistributiontheorytwo}
    \|\Delta^{k}f\|_{\infty}
    \leq C(k) (r_{i}\mu_{i})^{-m_{0}k} \biggl( \sum_{l=0}^{L} (r_{i}\mu_{i})^{m_{0}l} |\rho_{l}| + \sum_{l=0}^{L-1}
    r_{i}^{m_{0}(l+1)}\mu_{i}^{m_{0}l} |\sigma_{l}| \biggr) +\epsilon.
    \end{equation}
Similar estimates hold for other norms and seminorms, provided only that their scaling is
understood.  For example, if we consider the seminorm $\DF(\Delta^{k}f)^{1/2}$, then the scaling in
\eqref{Borel-scalingestforflm} and \eqref{Borel-scalingestforglm} is reduced by $r_{i}^{-m/2}$, so
the above construction would give
\begin{equation*}
    \DF(\Delta^{k}f)^{1/2}
    \leq C(k) r_{i}^{-m_{0}(k+1/2)} \mu_{i}^{-m_{0}k} \biggl( \sum_{l=0}^{L} (r_{i}\mu_{i})^{m_{0}l} |\rho_{l}| + \sum_{l=0}^{L-1}
    r_{i}^{m_{0}(l+1)}\mu_{i}^{m_{0}l} |\sigma_{l}| \biggr) +\epsilon.
    \end{equation*}

The result of the theorem may be transferred to any junction point $F_{w}(q)$ and cell $F_{w}(X)$
in $X$, simply by modifying the desired jet to account for the effect of composition with $F_{w}$,
solving for $f$ on $X$, and defining the new function to be $f\circ F_{w}^{-1}$ on the cell.  We
record a version of this that will be useful later; note that in the following we use the notation
$\partial_{n}$ for the normal derivative with respect to the cell $F_{w}(X)$.
\begin{corollary}\label{Borel-borelthmatjunctionpt}
    Let $F_{w}(q)$ be a junction point in $X$.  Given a jet $(\rho_{0},\rho_{1},\dotsc)$,
    $(\sigma_{1},\sigma_{2},\dotsc)$ there is a smooth function $f$ on $F_{w}(X)$ that has
    $\Delta^{k}f(p)=\partial_{n}\Delta^{k}f(p)=0$ at all points $p\in\partial F_{w}(X)$
    such that $p\neq F_{w}(q)$, and satisfies $\Delta^{k}f(F_{w}(q))=\rho_{k}$ and
    $\partial_{n}\Delta^{k}f(F_{w}(q))=\sigma_{k}$ for all $k$.
    \end{corollary}


\section{Additive Partitions of Functions}\label{Partition-section}

The results of \cite{MR1707752} show that multiplication is not generally a good operation on
functions in $\dom(\Delta)$.  In particular, for $X$ a p.c.f. fractal with self-similar measure and
regular self-similar harmonic structure it is generically the case that if $u\in\dom(\Delta)$ then
$u^{2}\not\in\dom(\Delta)$.  In such a situation there is no hope of using a smooth partition of
unity to localize problems in the classical manner.  Instead we provide a simple method for making
a smooth decomposition of $f\in\dom(\Delta^{\infty})$ using Theorem \ref{Borel-borelthm}.
Throughout this section we make the same assumptions on $X$ as were made in
Section~\ref{Borel_section}.
\begin{theorem}\label{Partition-partitionthm}
Let $\cup_{\alpha} \Omega_{\alpha}$ be an open cover of $X$ and $f\in\dom(\Delta^{\infty})$.  There
is a decomposition $f=\sum_{k=1}^{K}f_{k}$ in which each $k$ has a corresponding $\alpha_{k}$ such
that $f_{k}$ is smooth on $X$ and supported in $\Omega_{\alpha_{k}}$.
\end{theorem}
\begin{proof}
Compactness of $X$ allows us to reduce to the case of a finite cover
$\cup_{1}^{K}\Omega_{\alpha_{k}}$ for which there is no sub-collection that covers $X$.  We write
$\Omega_{k}$ for $\Omega_{\alpha_{k}}$, and construct the functions $f_{k}$ inductively.  At the
$k$-th stage we suppose there are functions $f_{1},\dotsc,f_{k-1}$ with the properties asserted in
the lemma, and that function $g_{k-1}= f-\sum_{l=1}^{k-1}f_{l}$ is smooth on $X$ and vanishes
identically on a neighborhood $\Pi_{k-1}$ of $X\setminus\Bigl(\cup_{j=k}^{K}\Omega_{j}\Bigr)$. In
the base case $k=1$ this assumption is trivial, and it is clear that the theorem follows
immediately from the induction.  We have therefore reduced to the case where our cover consists of
the two sets $\Omega_{k}$ and $\tilde{\Omega}_{k}=\cup_{j=k+1}^{K}\Omega_{j}$, because the
induction is complete once we have $f_{k}$ as in the lemma such that $g_{k}$ is identically zero on
a neighborhood $\Pi_{k}$ of $X\setminus\tilde{\Omega}_{k}$.

For a scale $m$ and $x\in X$, define the $m$-scale open neighborhood of $x\,$ to be the interior of
the unique $m$-cell containing $x$ if $x\not\in V_{m}$, and to be the union
$\{x\}\cup\bigl(\cup_{w} F_{w}(X)\setminus \partial F_{w}(X)\bigr)$ if $x=F_{w}(q_{i})$ is a
junction point. By
 Section 1.3 of \cite{Kigamibook}, the $m$-scale open neighborhoods form a fundamental system of open
neighborhoods of $x$. At each $x$ in $\Omega_{k}$ there is a largest $m$ such that the $m$-scale
neighborhood of $x$ is contained in $\Omega_{k}$. The collection of all such largest neighborhoods
of points of $\Omega_{k}$ is an open cover of the compact set
$\sppt(g_{k-1})\setminus\tilde{\Omega}_{k}$.  We use $\Lambda_{k}$ to denote the union over a
finite subcover.

Clearly $\Lambda_{k}$ has finitely many boundary points. Let those boundary points that are also in
$\tilde{\Omega}_{k}$ be $x_{1},\dotsc,x_{J}$, and take at each a finite collection of cells
$\{C_{i,j}\}_{i=1}^{I_{j}}$ having $x_{j}$ in their boundary.  We require that
$\Lambda_{k}\cup\Bigl(\cup_{i=1}^{I_{j}} C_{i,j}\Bigr)$ contains a neighborhood of $x_{j}$, that
all of the $C_{i,j}$ lie entirely within $\Omega_{k}$ and none intersect $\Lambda_{k}$, and that
$C_{i,j}\cap C_{i^{\prime},j^{\prime}}$ is empty unless $j=j^{\prime}$, in which case it contains
only $x_{j}$. On each cell we apply Corollary \ref{Borel-borelthmatjunctionpt} to find functions
$h_{i,j}$ that match $g_{k-1}(x_{j})$ and all powers of its Laplacian at $x_{j}$, and such that the
sum $\sum_{i}h_{i,j}$ has normal derivatives that cancel $\partial_{n}\Delta^{n}g_{k-1}(x_{j})$ at
$x_{j}$ for all $n$.  Thus $\sum_{i=1}^{I_{j}}h_{i,j}$ matches $g_{k-1}$ in the sense of the
matching condition and vanishes to infinite order at the other boundary points of
$\cup_{i=1}^{I_{j}}C_{i,j}$.  The matching condition implies that $f_{k}=g_{k-1}|_{\Lambda_{k}}
+\sum_{j}\sum_{i} h_{i,j}$ is smooth. It is clearly supported on $\Omega_{k}$ and equal to $f$ on
the closure $\overline{\Lambda}_{k}$ of $\Lambda_{k}$, so $g_{k}=g_{k-1}-f_{k}$ is zero on a
neighborhood $\Pi_{k}$ of $X\setminus \tilde{\Omega}_{k}$, which completes the induction and the
proof.
\end{proof}

\begin{remark}
Theorem~\ref{Partition-partitionthm} is stronger than the partitioning results we obtained directly
from the heat kernel smoothing procedure of Section~\ref{Heatbump-section}, but only applies to the
more limited case of p.c.f. self-similar sets with regular Dirichlet form.  We note, however, that
it provides a stronger version of Theorem~\ref{FixedPoint-smoothbumpexiststheorem}, because by
smoothly cutting off the constant function $1$ we obtain a smooth bump that is identically $1$ on a
compact $K$ and $0$ outside an $\epsilon$-neighborhood of $K$.  In contrast to
Corollary~\ref{Heatbump-heatbumptheorem} it does not imply that there is a positive smooth bump of
this type.
\end{remark}

\def\cprime{$'$}
\providecommand{\bysame}{\leavevmode\hbox to3em{\hrulefill}\thinspace}
\providecommand{\MR}{\relax\ifhmode\unskip\space\fi MR }
\providecommand{\MRhref}[2]{%
  \href{http://www.ams.org/mathscinet-getitem?mr=#1}{#2}
}
\providecommand{\href}[2]{#2}

\end{document}